\documentclass[11pt]{article}

\usepackage{amsmath,amssymb,amsfonts,mathtools}
\usepackage{amsthm}
\usepackage{graphicx}
\usepackage{textcomp}
\usepackage{braket}
\usepackage{xcolor}
\usepackage{a4wide}
\usepackage{authblk}
\usepackage{bbm}

\usepackage{hyperref}
\usepackage{algorithmic}
\usepackage{todonotes}

\usepackage[ruled,vlined, linesnumbered]{algorithm2e}
\newtheorem{thm}{Theorem}
\newtheorem{cor}{Corollary}

\newtheorem{prop}{Proposition}
\newtheorem{lem}{Lemma}

\newcommand{\Rset}{\mathbb{R}}

\newcommand{\randxi}{\tilde{\pmb{\xi}}}
\newcommand{\randxiT}{\tilde{\pmb{\xi}}{}^T}

\begin{document}

\title{Wasserstein robust combinatorial optimization problems} 

\author[1]{Marcel Jackiewicz}
\author[1]{Adam Kasperski}
\author[1]{Pawe{\l} Zieli\'nski\footnote{Corresponding author}}

\affil[1]{
Wroc{\l}aw  University of Science and Technology, Wroc{\l}aw, Poland\\
            \texttt{\{marcel.jackiewicz,adam.kasperski,pawel.zielinski\}@pwr.edu.pl}}

\date{}

\maketitle

\begin{abstract}
This paper discusses a class of combinatorial optimization problems with uncertain costs in the objective function. It is assumed that a sample of the cost realizations is available, which defines an empirical probability distribution for the random cost vector. A Wasserstein ball, centered at the empirical distribution, is used to define an ambiguity set of probability distributions. A solution minimizing the Conditional Value at Risk for a worst probability distribution in the Wasserstein ball is computed. The complexity of the problem is investigated. Exact and approximate solution methods for various support sets are proposed. Some known results for the Wasserstein robust shortest path problem are generalized and refined.
\end{abstract}


\section{Introduction}

In a traditional combinatorial optimization problem, we seek an object composed of some elements of a ground set, which minimizes a linear objective function. Unfortunately, most combinatorial optimization problems are NP-hard, except for some special cases such as fundamental network problems (see, for example~\cite{AMO93, PS98}). In many practical applications, the exact values of the element costs are unknown. The method of handling this uncertainty mainly depends on the knowledge available. In many situations, it is only possible to specify an uncertainty set $\mathcal{U}$ which contains possible values of uncertain costs. 
Then, some robust criteria, such as the min-max or the min-max regret, can be used to compute a solution~\cite{KY97}. However, additional knowledge in the scenario set is available in many situations. For example, some bounds on mean and covariance matrix~\cite{DY10}, or a family of confidence sets for the random cost vector can be specified~\cite{WKS14}. Also, a sample of realizations of the random cost vector can be available. In all of these cases, one can define a set of admissible probability distributions for the random cost vector.  Using a robust distributionally framework, we can minimize some risk measure for the worst probability distribution in this set.  A risk measure commonly used in stochastic optimization is the Conditional Value at Risk (CVaR)~\cite{P00, RU00}. The CVaR is a convex coherent risk measure that has a nice interpretation and possesses many favorable computational properties. 

When a sample of the random cost vector is available, then the Wasserstein metric~\cite{KR58} can be used to define a set of admissible probability distributions. The Wasserstein metric defines a distance between two probability distributions on the same support set. A Wasserstein ball contains then all probability distributions that are of some prescribed distance to the empirical distribution, which is defined by the available sample of the random cost vector. The Wasserstein ball contains all admissible probability distributions, and the robust distributionally framework can be used to compute a solution~\cite{KE18, GK23}. The approach based on the Wasserstein distance can be attractive when the sample size is small. In particular, it may outperform the approach based on the sample average approximation~\cite{KE18}.

 Our paper is motivated by the paper by Wang et al.~\cite{WYSZ19}, in which the Wasserstein robust approach has been applied to the shortest path problem. The authors provided in~\cite{WYSZ19} various mathematical programming formulations and showed the results of some computational tests. In this paper, we generalize the results obtained in~\cite{WYSZ19} to all combinatorial problems. In particular, we generalize the method of computing a worst probability distribution for a given solution and the mathematical programming formulation for the unrestricted support set. We also consider the problem with more restricted support sets. In the simplest case, lower and upper bounds for the costs are known, and the support set is a hyperrectangle  in~$\Rset^n_{+}$. We refine the results for such a support set shown in~\cite{WYSZ19}. Generally, any closed, convex, and bounded support set (for example, an ellipsoid in $\Rset^n_{+}$) can be used. For this case, we propose a row generation algorithm, which can be applied when the sample size is not large. We also propose a more efficient method of constructing a solution consisting of solving the CVaR-optimization problem for an appropriately distorted sample. We show that this algorithm can provide an approximate solution to the robust Wasserstein problem. Finally, we also investigate the computational complexity of the problem. We show that it is strongly NP-hard for most basic problems, particularly for the shortest path problem discussed in~\cite{WYSZ19}. We also present the results of some computational tests. This paper is a significantly extended version of the results presented in~\cite{JKZ23}.

 This paper is organized as follows. In Section~\ref{sec1}, we present the problem formulation. We also show how the problem can be reformulated in a more tractable way, based on a characterization of the worst probability distribution. In Section~\ref{sec3}, we investigate the complexity of the problem. We prove that the Wasserstein robust problem is strongly NP-hard for most basic problems (particularly for the shortest path), except when CVaR is the expectation. In Section~\ref{sec4}, we propose various methods of solving the problem. In Section~\ref{sec4_1}, we discuss the case of unrestricted support set, i.e. when any realization of cost in~$\Rset^n_{+}$ can occur. We show a mathematical programming formulation for this case and propose some exact and approximate methods of solving it. In Section~\ref{sec4_2}, we consider a simple case of restricted support being a hyperrectangle in $\Rset^n_{+}$. We refine the known results for this support set. In Section~\ref{sec4_3}, we consider the problem with an arbitrary closed, convex, and bounded support set, and we propose a row generation algorithm for that case. In Section~\ref{sec4_4}, we construct a more efficient approximation algorithm and prove its approximation guarantee. Finally, Section~\ref{sec5} contains the results of some computational tests for the 0-1 Knapsack problem.

\section{Problem formulation}
\label{sec1}

Consider the following combinatorial optimization problem:
$$\mathcal{P}: \;\; \min_{\pmb{x}\in \mathcal{X}} \;\pmb{\xi}^T\pmb{x},$$
where $\mathcal{X}\subseteq \{0,1\}^n$ is a set of feasible solutions and $\pmb{\xi}\in \Rset_{+}^n$ is a given nonnegative cost vector. Typically,~$\mathcal{X}$ is described by a system of linear constraints of the form $\pmb{A}\pmb{x}\geq \pmb{b}$, for some fixed matrix $\pmb{A}\in \Rset^{m\times n}$ and vector $\pmb{b}\in \Rset^m$, with $\pmb{x}\in \{0,1\}^n$.
 In general, problem $\mathcal{P}$ is NP-hard, and only some special cases are known to be polynomially solvable. One such a particular case is when the matrix $~\pmb{A}$ is totally unimodular, for example, when $\mathcal{P}$ is the \textsc{Shortest Path} or the \textsc{Minimum Assignment}  (see~\cite{AMO93, PS98}).

In practice, the true realization of the cost vector $\pmb{\xi}$ is often unknown before the optimal solution $\pmb{x}$ to $\mathcal{P}$ must be determined.  It should be rather regarded as  a random vector
$\randxi$ with a probability distribution $\mathbb{P}$ supported on $\Xi\subseteq \Rset^n_+$. In this paper, we assume that the support set $\Xi$ is convex. If $\Xi=\Rset^n_{+}$, then we say that the support set is unrestricted. However, $\Xi$ can be restricted to a closed and bounded set in many practical situations. In this case, $\Xi$ coincides with an uncertainty set commonly used in robust optimization~\cite{BN09}.
Vector $\randxi$ induces a random cost $\randxiT\pmb{x}$ of  solution $\pmb{x}\in \mathcal{X}$. To choose a solution, a risk measure that assigns a real value to random solution cost should be used~\cite{KW94}. In this paper, we will apply the following Conditional Value at Risk (Expected Shortfall)~\cite{P00,RU00}, which is commonly used in stochastic optimization (see, e.g.~\cite{KW94}):
\begin{equation}
\label{cvare}
{\rm CVaR}^{\alpha}_{\mathbb{P}}[\randxiT\pmb{x}]=\inf\left\{t\in \Rset : t + \frac{1}{\alpha} {\rm E}_{\mathbb{P}}([\randxiT\pmb{x}-t]_{+})\right\}=\inf\left\{t \in \Rset : t + \frac{1}{\alpha} \int_{\Xi}[\pmb{\xi}^T\pmb{x}-t]_{+}\mathbb{P}({\rm d}\pmb{\xi})\right\},
\end{equation}
where $\alpha\in (0,1]$ is a fixed risk level,
 ${\rm E}_{\mathbb{P}}[\cdot]$ is the expected value with respect to the probability distribution~$\mathbb{P}$
 and $[x]_{+}=\max\{0,x\}$.
Note that for $\alpha=1$, the Conditional Value at Risk is the expected solution cost ${\rm E}_{\mathbb{P}}[\randxiT\pmb{x}]$. 
Given a probability distribution $\mathbb{P}$ for $\randxi$, we seek a solution minimizing the Conditional Value at Risk at a given risk level $\alpha\in (0,1]$, that is, we solve the following problem:
\begin{equation}
\label{mincv}
\text{CVaR}~\mathcal{P}:\;\;
\begin{array}{lll}
	\min & {\rm CVaR}^{\alpha}_{\mathbb{P}}[\randxiT\pmb{x}]\\
		& \pmb{A}\pmb{x}\geq \pmb{b}, \\
		& \pmb{x}\in \{0,1\}^n.
	\end{array}
\end{equation}

In practical applications, the true probability distribution $\mathbb{P}$ in problem~(\ref{mincv}) is rarely known exactly. Suppose that we have sample $\hat{\pmb{\xi}}_1,\dots,\hat{\pmb{\xi}}_N$ of $\randxi$, which induces the empirical probability distribution 
\begin{equation}
\label{dird}
\widehat{\mathbb{P}}^N:=\frac{1}{N}\sum_{i=1}^N \delta_{\hat{\pmb{\xi}}_i},
\end{equation}
where $\delta_{\pmb{\xi}_i}$ is a Dirac distribution concentrating unit mass at $\pmb{\xi}_i\in\Xi$. In this paper we will use a more convenient notation $\widehat{\mathbb{P}}^N:=(\hat{\pmb{\xi}}_1,\dots,\hat{\pmb{\xi}}_N)$, so we will identify $\widehat{\mathbb{P}}^N$ with a vector of realizations $(\hat{\pmb{\xi}}
_1,\dots,\hat{\pmb{\xi}}_N)$. We will generally use the notation $\mathbb{P}^N:=(\pmb{\xi}_1,\dots,\pmb{\xi}_N)$ to denote a discrete probability distributions in $\Xi$, defined analogously as~(\ref{dird}). The indexed vectors $\pmb{\xi}_1,\dots,\pmb{\xi}_N\in \Rset^n_{+}$ are realizations of the random vector $\randxi$. We point out that some realizations in $\mathbb{P}^N$ can be repeated, which leads to a nonuniform discrete probability distribution.

Given two probability distribution $\mathbb{Q}_1$ and $\mathbb{Q}_2$ supported on $\Xi$, we define the Wasserstein distance between them as follows~\cite{KR58}:
	$$d_W(\mathbb{Q}_1,\mathbb{Q}_2)=\inf \left\{\int_{\Xi \times \Xi} ||\pmb{\xi}_1-\pmb{\xi}_2||_q \Pi({\rm d} \pmb{\xi}_1, {\rm d}\pmb{\xi}_2): \overset{\Pi \text{ is a joint probability distribution in }\Xi \times \Xi}{\footnotesize \text{with marginals } \tiny \mathbb{Q}_1, \mathbb{Q}_2} \right\},$$
	where $||\pmb{v}||_q=(\sum_{i=1}^n |v_i|^q)^{1/q}$ is the $q$-norm in $\Rset^n$ for $q\in \overline{\Rset}_{\geq 1}=
	\Rset_{\geq 0}\cup\{\infty\}$.
Using the Wasserstein distance, we can define the following ambiguity set of probability distributions (also called a Wasserstein ball)~\cite{KE18}:
$$\mathbb{B}_{\epsilon}(\widehat{\mathbb{P}}^N)=\{\mathbb{Q}\in \mathcal{M}(\Xi): d_W(\widehat{\mathbb{P}}^N,\mathbb{Q})\leq \epsilon\},$$
where $\mathcal{M}(\Xi)$ is the set of all probability distributions in $\Xi$ and $\epsilon\geq 0$ is a given constant. The following theorem provides the  foundations for establishing a finite sample guarantees:
\begin{thm}[\cite{KE18, FG14}]
	Assume that there exists $a>1$ such that
		$$A=\int_{\Xi}{\rm exp}(||\pmb{\xi}||^a)\mathbb{P}({\rm d} \pmb{\xi})<\infty.$$
	Then the true probability distribution $\mathbb{P}$ belongs to $\mathbb{B}_{\epsilon}(\widehat{\mathbb{P}}^N)$ with the confidence $1-\beta$ for a fixed $\beta\in(0,1)$, where
	$$
		\epsilon=\left\{\begin{array}{llll}
							\left(\frac{\log(c_1 \beta^{-1})}{c_2N}\right)^{1/\max\{n,2\}} & \text{if} & N\geq \frac{\log(c_1\beta^{-1}}{c_2}),\\
							\left(\frac{\log(c_1 \beta^{-1})}{c_2N}\right)^{1/a} & \text{if} & N< \frac{\log(c_1\beta^{-1}}{c_2}),
					\end{array}
						\right.
	$$
	and $c_1,c_2$ are positive constants that depend on $A$, $a$ and $n$.
	
\end{thm}

Given an empirical distribution $\widehat{\mathbb{P}}^N$ and $\epsilon\geq 0$, we consider the following distributionally robust problem~\cite{KE18}:

\begin{equation}
\label{mincvd00}
\textsc{Distr}~\mathcal{P}:\;\;
\begin{array}{lll}
	\min & \displaystyle \sup_{\mathbb{P}\in \mathbb{B}_{\epsilon}(\widehat{\mathbb{P}}^N)} {\rm CVaR}^{\alpha}_{\mathbb{P}}[\randxiT\pmb{x}]\\
		& \pmb{A}\pmb{x}\geq \pmb{b}, \\
		& \pmb{x}\in \{0,1\}^n.
	\end{array}
\end{equation}
If the supremum in~(\ref{mincvd00}) is attained by some probability distribution $\mathbb{P}\in \mathbb{B}_{\epsilon}(\widehat{\mathbb{P}}^N)$, then $\mathbb{P}$ is called a worst probability distribution for $\pmb{x}$.
 In this paper, we will use the following lemma, which was originally proven for the \textsc{Shortest Path} problem 
 in~\cite[Lemma~2]{WYSZ19}. We extend the result~\cite{WYSZ19} for 
 a general combinatorial problem $\mathcal{P}$.
 \begin{lem}
\label{lemworst}
For each solution $\pmb{x}\in \mathcal{X}$, there is a worst probability distribution which belongs to
$$\mathcal{B}_{\epsilon}(\widehat{\mathbb{P}}^N)=\{\mathbb{P}^N=(\pmb{\xi}_1,\dots,\pmb{\xi}_N): \sum_{i=1}^N ||\pmb{\xi}_i-\hat{\pmb{\xi}}_i||_q\leq N\epsilon,\; \pmb{\xi}_i\in \Xi, i\in [N]\},$$
where $[N]$ denotes the set $\{1,\dots,n\}$.
\end{lem} 
\begin{proof}
See Appendix~\ref{dod}.
\end{proof}
Lemma~\ref{lemworst} shows that the worst probability distribution for a given solution $\pmb{x}$ is a discrete one and belongs to the closed, bounded and convex set $\mathcal{B}_{\epsilon}(\widehat{\mathbb{P}}^N)$. The set $\mathcal{B}_{\epsilon}(\widehat{\mathbb{P}}^N)$ has a nice description, which we will exploit in the next part of the paper. Using Lemma~\ref{lemworst}, we can represent~(\ref{mincvd00}) as the following optimization problem:

\begin{equation}
\label{pf}
\textsc{Distr}~\mathcal{P}:\;\;
\begin{array}{lll}
	\min & \displaystyle \max_{\mathbb{P}^N\in\mathcal{B}_{\epsilon}(\widehat{\mathbb{P}}^N)} {\rm CVaR}^{\alpha}_{\mathbb{P}^N}[\randxiT\pmb{x}]\\
		& \pmb{A}\pmb{x}\geq \pmb{b}, \\
		& \pmb{x}\in \{0,1\}^n.
	\end{array}
\end{equation}
 
 In the next part of this paper, we will investigate the properties of the $\textsc{Distr}~\mathcal{P}$ problem, using the formulation~(\ref{pf}).

\section{Minimizing the Conditional Value at Risk}
\label{sec3}

Observe that $\textsc{Distr}~\mathcal{P}$ reduces to $\textsc{CVaR}~\mathcal{P}$ for $\mathbb{P}=\widehat{\mathbb{P}}^N$ when $\epsilon=0$.
In this section we will focus on solving $\textsc{CVaR}~\mathcal{P}$, i.e. on minimizing the Conditional Value at Risk for a given discrete probability distribution $\mathbb{P}^N=(\pmb{\xi}_1,\dots,\pmb{\xi}_N)$. We will show some complexity and approximation results, which will be used to characterize the computational properties of the more general $\textsc{Distr}~\mathcal{P}$ problem.  The $\textsc{CVaR}~\mathcal{P}$ problem for $\mathbb{P}^N=(\pmb{\xi}_1,\dots,\pmb{\xi}_N)$ can be represented as the following mixed integer program:
 \begin{equation}
\label{mincv1}
\begin{array}{lll}
	\min & \displaystyle t+\frac{1}{\alpha N}\sum_{i=1}^N u_i\\
		& \pmb{\xi}_i^T\pmb{x}-t\leq u_i, & i\in [N],\\
		& \pmb{A}\pmb{x}\geq \pmb{b}, \\
		& \pmb{x}\in \{0,1\}^n,\\
		& u_i\geq 0 & i\in [N].
	\end{array}
\end{equation}

Here and subsequently, we will make use of the fact that the Conditional Value at Risk is a special case of the Ordered Averaging Aggregation (OWA) criterion introduced in~\cite{YA88}. Let $\pmb{w}=(w_1,\dots,w_N)$ be a vector of nonnegative weights such that $w_1+\dots+w_N=1$. The OWA criterion for $\mathbb{P}^N=(\pmb{\xi}_1,\dots,\pmb{\xi}_N)$  is defined as follows:
\begin{equation}
\label{owadef}
{\rm Owa}^{\pmb{w}}_{\mathbb{P}^N}[\randxiT\pmb{x}]=\sum_{i=1}^N w_i \pmb{\xi}_{\sigma(i)}^T\pmb{x},
\end{equation}
where $\sigma$ is a permutation of~$[N]$ such that
$\pmb{\xi}^T_{\sigma(1)}\pmb{x}\geq \pmb{\xi}^T_{\sigma(2)}\pmb{x}\geq \dots\geq \pmb{\xi}^T_{\sigma(N)}\pmb{x}$.

\begin{prop}
\label{propowa}
Assume that $\alpha\in (\frac{l-1}{N}, \frac{l}{N}]$ for $l\in [N]$. Then
$${\rm CVaR}^{\alpha}_{\mathbb{P}^N}[\randxiT\pmb{x}]={\rm Owa}^{\pmb{w}}_{\mathbb{P}^N}[\randxiT\pmb{x}],$$
where 
\begin{equation}
\label{owaw}
\pmb{w}=\left\{\begin{array}{lll}
	(1,0,\dots,0) & \text{if} & \alpha< \frac{1}{N}, \\
(\underbrace{\frac{1}{\alpha N}, \dots, \frac{1}{\alpha N}, 1-\frac{l-1}{\alpha N}}_{l \text{ positive weights}},\underbrace{0,\dots,0}_{N-l \text{ weights}}) & \text{if} & \alpha\geq \frac{1}{N}.
\end{array}\right.
\end{equation}
\end{prop}
\begin{proof}
	Fix $\pmb{x}\in \mathcal{X}$ and $\mathbb{P}^N=(\pmb{\xi}_1,\dots,\pmb{\xi}_N)$. Assume that $\pmb{\xi}^T_{1}\pmb{x}\geq \pmb{\xi}^T_{2}\pmb{x}\geq \dots\geq \pmb{\xi}^T_{N}\pmb{x}$.
	Using~(\ref{mincv1}), let us rewrite ${\rm CVaR}^{\alpha}_{\mathbb{P}^N}[\randxiT\pmb{x}]$ as the following linear programming problem:
	\begin{equation}
	\label{f00}
		\begin{array}{llll}
			\min &  \displaystyle t + \frac{1}{\alpha N}\sum_{i=1}^N u_i \\
				& u_i+t\geq \pmb{\xi}_i^T\pmb{x}, & i\in [N],\\
				& u_i\geq 0, & i\in [N].
		\end{array}
	\end{equation}
	The dual to~(\ref{f00}) with dual variables $w_i$, $i\in [N]$, is
	\begin{equation}
	\label{f01}
		\begin{array}{llll}
			\max &  \displaystyle \sum_{i=1}^N \pmb{\xi}^T_i\pmb{x} w_i \\
				& \displaystyle \sum_{i=1}^N w_i=1,\\
				& 0\leq w_i\leq \frac{1}{\alpha N}, & i\in [N].
		\end{array}
	\end{equation}
	The dual~(\ref{f01}) can be solved using the following greedy method. 
	If $\frac{1}{\alpha N}> 1$ (so $\alpha< \frac{1}{N}$),  then we set $w_{1}=1$ and $w_{i}=0$ for $i=2,\dots,N$. If $\frac{1}{\alpha N}\leq 1$, then we set $w_{i}=\frac{1}{N\alpha}$ for $i=1,\dots,l-1$, $w_{l}=1-\frac{l-1}{\alpha N}$, and $w_{i}=0$ for $i=l+1,\dots,N$. Hence, the optimal objective value of the dual equals~(\ref{owadef}) with the weights defined as~(\ref{owaw}). By the linear programming strong duality, the proposition follows.
\end{proof}
Observe that if $\alpha=\frac{l}{N}$ for some $l=1,\dots,N$, then ${\rm CVaR}^{\alpha}_{\mathbb{P}^N}[\randxiT\pmb{x}]$ is just the average of the $l$ largest solution costs under $\mathbb{P}^N=(\pmb{\xi}_1,\dots,\pmb{\xi}_N)$.
Using Proposition~\ref{propowa}, we now characterize the complexity of $\textsc{CVaR}~\mathcal{P}$ for some generic combinatorial optimization problem $\mathcal{P}$.
Namely,
 assume that $\mathcal{P}$ is the following \textsc{Representatives Selection} problem (\textsc{RS} for short), in which 
  we are given a set $\mathcal{T}$ of $n$ tools, numbered from $1$ to $n$. This set is partitioned into $\ell$ disjoint sets $T_1,\ldots,T_\ell$.
  We wish to choose a subset of the tools~$\mathcal{T}$  that contains exactly one tool from each set $T_i$, $i\in[\ell]$. In the deterministic case, each tool $j\in T$ has a nonnegative cost $\xi_j$, and we seek a subset of the tools (a solution) whose total cost is minimum. 
  Clearly, the set of feaslible solutions~$\mathcal{X}_s$ can be described by 
  the constraints $\sum_{j\in T_i} x_j=1$ for each $i\in[\ell]$.
  It is easily seen that the \textsc{RS} problem, in this case 
  has a trivial solution, which can be constructed by setting $x_{j_i}=1$, where $j_i$ is the index of the variable in $T_i$ with the smallest cost under $\pmb{\xi}=(\xi_1,\ldots, \xi_n)$.
    In the \textsc{Min-Max RS} problem, we are given a set $\pmb{\xi}_1,\dots,\pmb{\xi}_K\in \Rset_{+}^n$ of realizations of the cost vector and we seek a solution minimizing the largest cost, i.e. we solve the following problem:
$$\min_{\pmb{x}\in \mathcal{X}_s}\max_{i\in[K]} \pmb{\xi}_i^T\pmb{x}.$$

The \textsc{Min-Max RS} problem has been investigated in~\cite{DW13, DK12, KKZ15}. In particular, it has been shown that it is strongly NP-hard when $K$ is a part of the input. We consider now the \textsc{CVaR~RS} problem, i.e. the problem~(\ref{mincv1}) with $\mathcal{X}_s$ and a fixed $\alpha\in (0,1]$.

\begin{thm}
\label{thmcompl}
	The \textsc{CVaR~RS} problem is strongly NP-hard for each $\alpha\in (0,1)$ when $K$ is a part of the input.
\end{thm}
\begin{proof}

Given an instance $(n,\mathcal{T}, \pmb{\xi}_1,\dots,\pmb{\xi}_K)$ of \textsc{Min-Max RS} and $\alpha\in (0,1)$, we construct the corresponding instance $(n',\mathcal{T}',\mathbb{P}^N, \alpha)$ of \textsc{CVaR~RS} as follows. Let $n'=n+1$ and $\mathcal{T}'=\mathcal{T}\cup T_0$, where $T_0=\{n+1\}$. This forms a set $\mathcal{X}'_s\in \{0,1\}^{n+1}$. Observe that $x_{n+1}=1$ in any feasible solution $\pmb{x}\in \mathcal{X}'_s$.
Let $l=\lceil \frac{K \alpha}{1-\alpha} \rceil$ and fix $N=K+l-1$. Define the probability distribution
 $$\mathbb{P}^{N}=(\pmb{\xi}'_1,\dots,\pmb{\xi}'_K,\underbrace{\pmb{\zeta},\dots,\pmb{\zeta}}_{l-1 \text{ copies}}),$$
 where the costs of the variables $x_i$, $i\in [n]$, under $\pmb{\xi}'_i$ are the same as under $\pmb{\xi}_i$ and the cost of $x_{n+1}$ is~0; the costs of the variables  $x_i$, $i\in [n]$, under $\pmb{\zeta}$ are~0 and the cost of $x_{n+1}$ is equal to a large constant $M$ (we can choose, for example, $M=\sum_{i=1}^K||\pmb{\xi}_i||_{1}$). Using the equality $\lceil x \rceil = x +r$ for $r\in [0,1)$, we get
 $$\frac{l-1}{N}=\frac{\lceil \frac{K \alpha}{1-\alpha} \rceil-1}{K+\lceil \frac{K \alpha}{1-\alpha} \rceil-1}=\frac{\frac{K \alpha}{1-\alpha}+r-1}{K+ \frac{K \alpha}{1-\alpha}+r-1}=\frac{K\alpha+(1-\alpha)(r-1)}{K+(1-\alpha)(r-1)}< \alpha,$$
 where the inequality follows from the fact that $\alpha, r\in [0,1)$ and thus $(1-\alpha)(r-1)\in [-1,0)$.  We also obtain
 $$\frac{l}{N}=\frac{\lceil \frac{K \alpha}{1-\alpha} \rceil}{K+\lceil \frac{K \alpha}{1-\alpha} \rceil-1}\geq \frac{ \frac{K \alpha}{1-\alpha}}{K+\frac{K \alpha}{1-\alpha}}= \alpha$$
 Hence  $\alpha\in (\frac{l-1}{N}, \frac{l}{N}]$. Using Proposition~\ref{propowa}, we get for any solution $\pmb{x}\in \mathcal{X}'_s$:
$${\rm CVaR}^{\alpha}_{\mathbb{P}^{N}}[\randxiT\pmb{x}]=
\underbrace{\frac{1}{\alpha N}M +\dots+\frac{1}{\alpha N}M}_{l-1 \text{ times }}+(1-\frac{l-1}{\alpha N})\pmb{\xi}_{\sigma(1)}^T\pmb{x},$$
where $\pmb{\xi}_{\sigma(1)}^T\pmb{x}=\max_{i\in [K]}\pmb{\xi}_i^T\pmb{x}$. Note that $(1-\frac{l-1}{\alpha N})>0$, since there are $l$ positive weights in~(\ref{owaw}) (see Proposition~\ref{propowa}).
Thus, there is a solution to \textsc{Min-Max RS} with the maximum cost at most $c$ if and only if there is a solution to \textsc{CVaR RS} with the Conditional Value at Risk at most $\frac{(l-1)M}{\alpha N}+(1-\frac{l-1}{\alpha N})c$.
\end{proof}

Since for $\epsilon=0$,  $\textsc{Distr}~\mathcal{P}$ is equivalent to $\textsc{CVaR}~\mathcal{P}$, we get the following corollary:
\begin{cor}
	The \textsc{Distr RS} problem is strongly NP-hard for each $\alpha\in (0,1)$.
\end{cor}

The \textsc{RS} problem can be seen as a very special case of other basic combinatorial optimization problems. In particular, it is a special case of the \textsc{Shortest Path} problem. To see this, consider the
 series-parallel
 graph $G$ shown in Figure~\ref{fig1}.
 \begin{figure}[ht]
\centering
\includegraphics[height=2.5cm]{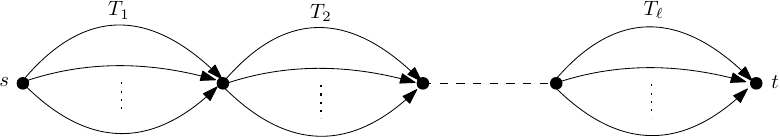}
\caption{An instance of the \text{Shortest Path} problem, corresponding to \textsc{RS}.}\label{fig1}
\end{figure}
The arcs belonging to the subsequent layers of $G$ correspond to $T_1,\dots, T_{\ell}$. The costs of the arcs are equal to the costs of the tools. It is easy to see that each $s-t$ path corresponds to some solution of the \textsc{RS} problem and vice versa. 
Since every $s-t$ path in $G$ is also a spanning tree of $G$, the same transformation applies to the \textsc{Minimum Spanning Tree} problem. This leads to the following corollary:
\begin{cor}
	The \textsc{Distr Shortest Path}  and \textsc{Distr Spanning Tree} problems are strongly NP-hard 
	in  series-parallel graphs
	for each $\alpha\in (0,1)$.
\end{cor}


The $\textsc{CVaR}~\mathcal{P}$ problem for $\alpha=1$ reduces to minimizing the expected solution cost. In this case, the problem with $\mathbb{P}^N=(\pmb{\xi}_1,\dots,\pmb{\xi}_N)$ easily reduces to solving the deterministic problem~$\mathcal{P}$ with the aggregated cost vector $\hat{\pmb{\xi}}=\frac{1}{N}\sum_{i=1}^N\pmb{\xi}_i$. Hence, the complexity of $\textsc{CVaR}~\mathcal{P}$ for $\alpha=1$ is then the same as the complexity of~$\mathcal{P}$. It has been shown in~\cite{KZ15} that solving~$\mathcal{P}$ for the aggregated cost vector $\hat{\pmb{\xi}}$ yields a $w_1N$-approximate solution to the problem of minimizing OWA with nonincreasing weights, i.e. $w_1\geq w_2\geq\dots\geq w_N$. 
Using Proposition~\ref{propowa} we conclude that $w_1=\min\{1,\frac{1}{\alpha N}\}$, so the approximation ratio becomes $\min\{N, \frac{1}{\alpha}\}$. We have thus established the following result:
\begin{prop}
If $\mathcal{P}$ can be solved in polynomial time, then $\textsc{CVaR}~\mathcal{P}$ with $\alpha\in (0,1)$  is approximable within $\min\{N,\frac{1}{\alpha}\}$.
\end{prop}

\section{Solving the Wasserstein distributionally robust problem}
\label{sec4}

In this section, we will propose some exact and approximate methods of solving the $\textsc{Distr}~\mathcal{P}$ problem. We will consider the problem with various   $q$-norms,
$q\in \overline{\Rset}_{\geq 1}$, and support sets~$\Xi$. We start with the case when $\Xi=\Rset^n_+$. Then, we will consider a support $\Xi$ being a hyperrectangle in $\Rset^n_+$ (such a model has been considered in~\cite{WYSZ19}). Finally, we will show some results for any convex, closed, and bounded support $\Xi$.

\subsection{The support set  $\Xi=\Rset^n_{+}$}
\label{sec4_1}

The following theorem is a generalization of the result obtained in~\cite{WYSZ19}:

\begin{thm}
\label{propunr}
	If $\Xi=\Rset^n_{+}$, then $\textsc{Distr}~\mathcal{P}$ can be expressed as
	\begin{equation}
\label{mincvd2}
\begin{array}{lll}
	\min & {\rm CVaR}^{\alpha}_{\widehat{\mathbb{P}}^N}[\randxiT\pmb{x}]+\gamma \epsilon ||\pmb{x}||_{q'}\\
		& \pmb{A}\pmb{x}\geq \pmb{b}, \\
		& \pmb{x}\in \{0,1\}^n,
	\end{array}
\end{equation}
where $\gamma=N$ if $\alpha<\frac{1}{N}$, $\gamma=\frac{1}{\alpha}$ if $\alpha\geq \frac{1}{N}$, and $||\cdot||_{q'}$ is the dual norm to $||\cdot||_{q}$, i.e. $\frac{1}{q}+\frac{1}{q'}=1$.
\end{thm}
\begin{proof}
See Appendix~\ref{dod}.
\end{proof}

If $q=1$, then~(\ref{mincvd2}) reduces to minimizing the Conditional Value at Risk for the empirical distribution $\widehat{\mathbb{P}}^N$. Indeed, the dual norm $||\pmb{x}||_{\infty}$ is either~0 (when $\pmb{x}=\pmb{0}$) or~1 (when $\pmb{x}\neq \pmb{0})$ and the second component in the objective function of~(\ref{mincvd2}) is constant. If $q=\infty$, then $||\pmb{x}||_{1}=\sum_{i=1}^n x_i$ and~(\ref{mincvd2}) becomes  a mixed integer program. Consider now the case of arbitrary 
$q\in \overline{\Rset}_{\geq 1}$. Since $||\pmb{x}||_{q'}=(\sum_{i=1}^n x_i)^{\frac{q}{q-1}}$, the model (\ref{mincvd2}) can be converted to
\begin{equation}
\label{mod00}
	\begin{array}{lll}
	\min & \displaystyle t+\frac{1}{\alpha N}\sum_{i=1}^N u_i + \lambda^{\frac{q-1}{q}}\gamma\epsilon\\
		& \hat{\pmb{\xi}}_i^T\pmb{x}-t\leq u_i, & i\in [N]\\
		& \pmb{A}\pmb{x}\geq \pmb{b}, \\
		& \displaystyle \sum_{i=1}^n x_i\leq \lambda,\\
		& \pmb{x}\in \{0,1\}^n,\\
		& u_i\geq 0, & i\in [N]\\
		&\lambda\geq 0.
	\end{array}
\end{equation}
In the formulation~(\ref{mod00}) we use~(\ref{mincv1}) to represent ${\rm CVaR}^{\alpha}_{\widehat{\mathbb{P}}^N}[\randxiT\pmb{x}]$. In any optimal solution to~(\ref{mod00}) we can fix $\lambda=\sum_{i=1}^n x_i$ and thus $\lambda=||\pmb{x}||_{q'}$. Hence the optimal value of $\lambda$ is integer and takes one of the values in $\Lambda=\{n_{\min},\dots,n_{\max}\}$, where $n_{\min}$ and $n_{\max}$ are the minimum and the maximum cardinality of solution $\pmb{x}\in \mathcal{X}$, respectively. Thus, we need to solve~(\ref{mod00}) for each fixed $\lambda\in\Lambda$ and choose a solution with the minimum value of the objective function.  Hence, ~(\ref{mod00}) can be solved by solving a family of mixed integer programs.
If $\alpha=1$, then model~(\ref{mod01}) reduces to (see also~\cite{KWFH22})
\begin{equation}
\label{mod01}
	\begin{array}{lll}
	\min & \displaystyle {\rm E}_{\widehat{\mathbb{P}}^N}[\randxiT\pmb{x}]+  \epsilon ||\pmb{x}||_{q'}=\hat{\pmb{\xi}}^T\pmb{x}+ \epsilon ||\pmb{x}||_{q'}\\
		& \pmb{A}\pmb{x}\geq \pmb{b}, \\
		& \pmb{x}\in \{0,1\}^n,
	\end{array}
\end{equation}
where $\hat{\pmb{\xi}}=\frac{1}{N}\sum_{i=1}^N \hat{\pmb{\xi}}_i$,
which can be represented equivalently as
 \begin{equation}
\label{mod03}
	\begin{array}{lll}
	\min & \displaystyle \hat{\pmb{\xi}}^T\pmb{x}+ \lambda^{\frac{q-1}{q}} \epsilon\\
		& \pmb{A}\pmb{x}\geq \pmb{b}, \\
		& \displaystyle \sum_{i=1}^n x_i\leq \lambda,\\
		& \pmb{x}\in \{0,1\}^n,\\
		&\lambda\geq 0.
	\end{array}
\end{equation}
Sometimes, ~(\ref{mod01}) can be solved in polynomial time. Indeed, when $\mathcal{P}$ is polynomially solvable and $q=1$ or $q=\infty$. In the former case the term $||\pmb{x}||_{q'}$ is constant, while in the latter case $||\pmb{x}||_{q'}=\sum_{i=1}^n x_i$, which leads to a linear objective function in~(\ref{mod01}). For arbitrary $q$, we can solve the formulation~(\ref{mod03}) for each $\lambda\in\Lambda$, which can also be done in polynomial time for some problems $\mathcal{P}$.
 For example, when $\mathcal{P}$ is the \textsc{Shortest Path}, then~(\ref{mod03}) can be solved by solving a family of the \textsc{Constrained Shortest Path} problems with all arc weights equal to~1 in the budget constraint $\sum_{i=1}^n x_i\leq \lambda$. Such a special case of this problem can be solved in polynomial time~\cite{H92}. In~\cite{KWFH22} and alternative method of solving~(\ref{mod01})  has been proposed. This method consists of solving a family of $n+1$ deterministic problems $\mathcal{P}$ with some suitably chosen cost vectors. However, our alternative approach is easier and also leads to efficient solution methods for some particular problems.

\begin{prop}
	If $(\ref{mod01})$ can be solved in polynomial time, then~(\ref{mincvd2}) is approximable within $\gamma=\min\{N,\frac{1}{\alpha}\}$
\end{prop}	
\begin{proof}
We will show first that for each $\alpha\in (0,1]$
\begin{equation}
\label{ecv}
{\rm E}_{\widehat{\mathbb{P}}^N}[\randxiT\pmb{x}]\leq {\rm CVaR}^{\alpha}_{\widehat{\mathbb{P}}^N}[\randxiT\pmb{x}]\leq \gamma {\rm E}_{\widehat{\mathbb{P}}^N}[\randxiT\pmb{x}].
\end{equation}
 The first inequality in~(\ref{ecv}) easily follows from the definition of the Conditional Value at Risk. To prove the second inequality, we will use Proposition~\ref{propowa}. If $\alpha<\frac{1}{N}$ (hence $N<\frac{1}{\alpha}$), then
 $${\rm CVaR}^{\alpha}_{\widehat{\mathbb{P}}^N}[\randxiT\pmb{x}]=\hat{\pmb{\xi}}_{\sigma(1)}^T\pmb{x}\leq \sum_{i=1}^N \hat{\pmb{\xi}}_i^T\pmb{x}= N\cdot  {\rm E}_{\widehat{\mathbb{P}}^N}[\randxiT\pmb{x}].$$
 If $\alpha\geq \frac{1}{N}$ (hence $N\geq \frac{1}{\alpha}$), then
 $${\rm CVaR}^{\alpha}_{\widehat{\mathbb{P}}^N}[\randxiT\pmb{x}]\leq \frac{1}{\alpha N}\sum_{i=1}^N \hat{\pmb{\xi}}_i^T\pmb{x}=\frac{1}{\alpha}  {\rm E}_{\widehat{\mathbb{P}}^N}[\randxiT\pmb{x}],$$
 which yields~(\ref{ecv}).
Let $\pmb{x}'$ be an optimal solution to~(\ref{mod01}) and let $\pmb{x}^*$ be an optimal solution to~(\ref{mincvd2}). Using~(\ref{ecv}) and the fact that $\gamma\geq 1$ we obtain

$$\frac{1}{\gamma}({\rm CVaR}^{\alpha}_{\widehat{\mathbb{P}}^N}[\randxiT\pmb{x}']+\gamma\epsilon ||\pmb{x}'||_{q'})\leq {\rm E}_{\widehat{\mathbb{P}}^N}[\randxiT\pmb{x}']+\epsilon ||\pmb{x}'||_{q'}\leq {\rm E}_{\widehat{\mathbb{P}}^N}[\randxiT\pmb{x}^*]+ \epsilon ||\pmb{x}^*||_{q'}\leq$$
$$  {\rm CVaR}^{\alpha}_{\widehat{\mathbb{P}}^N}[\randxiT\pmb{x}^*]+\gamma\epsilon||\pmb{x}^*||_{q'},$$
which completes the proof.
\end{proof}

\subsection{The support set  $\Xi_{\pmb{a},\pmb{b}}$ and $q=1$}
\label{sec4_2}

In this section, we examine the model investigated in~\cite{WYSZ19}, where the support set is defined as $\Xi_{\pmb{a},\pmb{b}}=\{\pmb{\xi}\in \Rset^n_{+}: \pmb{a}\leq \pmb{\xi} \leq \pmb{b}\}$, so lower $\pmb{a}\in \Rset_{+}^n$  and upper $\pmb{b}\in \Rset_{+}^n$ bounds on the costs are known. To solve $\textsc{Distr}~\mathcal{P}$ for this support set, some mixed integer programs have been constructed in~\cite{WYSZ19}. In this section, we will show how the programs can be simplified. Furthermore, a special case with $q=1$ can be solved more efficiently.

If $\hat{\pmb{\xi}}_i\in \Xi_{\pmb{a},\pmb{b}}$ for each $i\in [N]$, then an optimal solution to $\textsc{Distr}~\mathcal{P}$ will not change if we replace $\Xi_{\pmb{a},\pmb{b}}$ with $\Xi_{\pmb{0},\pmb{b}}$, that is we fix the lower bounds $\pmb{a}=\pmb{0}$. Indeed, a worst probability distribution $\mathbb{P}^N\in \mathcal{B}_{\epsilon}(\widehat{\mathbb{P}}^N)$ for $\pmb{x}$ is of the form
$\mathbb{P}^N=(\hat{\pmb{\xi}}_1+\pmb{\Delta}_1,\dots,\hat{\pmb{\xi}}_N+\pmb{\Delta}_N)$,
where $\pmb{\Delta}_i\in \Rset^n_{+}$ for $i\in [N]$ (all the components of $\pmb{\Delta}_i$ are nonnegative). Therefore $\mathbb{P}^N$ is still a worst probability distribution for $\pmb{x}$ if we fix $\pmb{a}=\pmb{0}$. 

\begin{prop}
\label{propab}
	If the support set is $\Xi_{\pmb{0},\pmb{b}}$, $\alpha=\frac{l}{N}$ for some $l\in [N]$ and $q=1$, then for each $\pmb{x}\in \mathcal{X}$
	\begin{equation}
		\label{case0}
		\max_{\mathbb{P}^N\in\mathcal{B}_{\epsilon}(\widehat{\mathbb{P}}^N)} {\rm CVaR}^{\alpha}_{\mathbb{P}^N}[\randxiT\pmb{x}]=\min\left\{\pmb{b}^T\pmb{x}, {\rm CVaR}^{\alpha}_{\widehat{\mathbb{P}}^{N}}[\randxiT\pmb{x}]+\frac{1}{l}N\epsilon\right\}.
	\end{equation}
\end{prop}
\begin{proof}
	Fix $\pmb{x}\in \mathcal{X}$ and assume that $\hat{\pmb{\xi}}_1^T\pmb{x}\geq\dots\geq\hat{\pmb{\xi}}_N^T\pmb{x}$. 
We show first that
	 \begin{equation}
		\label{case01}
		\max_{\mathbb{P}^N\in\mathcal{B}_{\epsilon}(\widehat{\mathbb{P}}^N)} {\rm CVaR}^{\alpha}_{\mathbb{P}^N}[\randxiT\pmb{x}]\leq\min\left\{\pmb{b}^T\pmb{x}, {\rm CVaR}^{\alpha}_{\widehat{\mathbb{P}}^{N}}[\randxiT\pmb{x}]+\frac{1}{l}N\epsilon\right\}.
	\end{equation}
Indeed, the inequality $\max_{\mathbb{P}^N\in\mathcal{B}_{\epsilon}(\widehat{\mathbb{P}}^N)} {\rm CVaR}^{\alpha}_{\mathbb{P}^N}[\randxiT\pmb{x}]\leq \pmb{b}^T\pmb{x}$ follows from the fact that $\pmb{b}^T\pmb{x}$ is the maximum cost of $\pmb{x}$ in $\Xi_{\pmb{0},\pmb{b}}$. Using Proposition~\ref{propowa}, we get for each $\mathbb{P}^N=(\hat{\pmb{\xi}}_1+\pmb{\Delta}_1,\dots, \hat{\pmb{\xi}}_N+\pmb{\Delta}_N)\in \mathcal{B}_{\epsilon}(\widehat{\mathbb{P}}^N)$
	 $$
		 {\rm CVaR}^{\alpha}_{\mathbb{P}^N}[\randxiT\pmb{x}] =\frac{1}{l}\sum_{i=1}^l (\hat{\pmb{\xi}}_{\sigma(i)}+\pmb{\Delta}_{\sigma(i)})^T\pmb{x}
		 \leq \sum_{i=1}^l \hat{\pmb{\xi}}_i^T\pmb{x}+\frac{1}{l}\sum_{i=1}^l ||\pmb{\Delta}_{\sigma(i)}||_1\leq {\rm CVaR}^{\alpha}_{\widehat{\mathbb{P}}^{N}}[\randxiT\pmb{x}]+\frac{1}{l}N\epsilon
	$$
and~(\ref{case01}) follows.  Consider a probability distribution  of the form
	$$\mathbb{P}^{'N}=(\hat{\pmb{\xi}}_1+\pmb{\Delta}_1,\dots,\hat{\pmb{\xi}}_l+\pmb{\Delta}_l,\hat{\pmb{\xi}}_{l+1},\dots,\hat{\pmb{\xi}}_N)\in \mathcal{B}_{\epsilon}(\widehat{\mathbb{P}}^N),$$
	where $\sum_{i=1}^l ||\pmb{\Delta}_i||_{1}\leq N\epsilon$,  $\pmb{\Delta}_i\geq \pmb{0}$, and $\Delta_{ij}=0$ if $x_j=0$ for $i\in [N]$, $j\in [n]$. Using Proposition~\ref{propowa}, we get
	$${\rm CVaR}^{\alpha}_{\mathbb{P}^{'N}}[\randxiT\pmb{x}]=\frac{1}{l}\sum_{i=1}^l (\hat{\pmb{\xi}}_{i}+\pmb{\Delta}_{i})^T\pmb{x}={\rm CVaR}^{\alpha}_{\widehat{\mathbb{P}}^{N}}[\randxiT\pmb{x}]+\frac{1}{l}\sum_{i=1}^l \pmb{\Delta}_{i}^T\pmb{x}={\rm CVaR}^{\alpha}_{\widehat{\mathbb{P}}^{N}}[\randxiT\pmb{x}]+\frac{1}{l}\sum_{i=1}^l ||\pmb{\Delta}_i||_{1}.$$ 
	Consider two cases:
	\begin{enumerate}
		\item  $\sum_{i=1}^l ||\pmb{\Delta}_i||_1<N\epsilon$. This means that either $(\hat{\pmb{\xi}}_i+\pmb{\Delta}_i)^T\pmb{x}=\pmb{b}^T\pmb{x}$ for each $i=1,\dots,l$  or we can further increase $\Delta_{ij}$ for some $x_j=1$ so that $\sum_{i=1}^l ||\pmb{\Delta}_i||_1=N\epsilon$ (and we go to Case~2). In the former case (see Proposition~\ref{propowa})
		$$ {\rm CVaR}^{\alpha}_{\mathbb{P}^{'N}}[\randxiT\pmb{x}]=\pmb{b}^T\pmb{x}\leq {\rm CVaR}^{\alpha}_{\widehat{\mathbb{P}}^{N}}[\randxiT\pmb{x}]+\frac{1}{l}N\epsilon. $$
		\item $\sum_{i=1}^l ||\pmb{\Delta}_i||_1=N\epsilon$. Then
		$$ {\rm CVaR}^{\alpha}_{\mathbb{P}^{'N}}[\randxiT\pmb{x}]={\rm CVaR}^{\alpha}_{\widehat{\mathbb{P}}^{N}}[\randxiT\pmb{x}]+\frac{1}{l}N\epsilon\leq \pmb{b}^T\pmb{x},$$	
	where the last inequality follows from the fact that $\pmb{b}^T\pmb{x}$ is an upper bound on $ {\rm CVaR}^{\alpha}_{\mathbb{P}^{'N}}[\randxiT\pmb{x}]$.
	\end{enumerate}
The above cases show that equality in~(\ref{case01}) holds.
\end{proof}
We can find an optimal solution to $\textsc{Distr}~\mathcal{P}$ for the problem satisfying the assumptions of Proposition~\ref{propab} by solving two easier problems. Let $\pmb{x}_1\in \mathcal{X}$ minimize $\pmb{b}^T\pmb{x}$ and $\pmb{x}_2\in \mathcal{X}$ minimize $ {\rm CVaR}^{\alpha}_{\widehat{\mathbb{P}}^{N}}[\randxiT\pmb{x}]$. Solution $\pmb{x}_1$ can be found be solving the deterministic problem $\mathcal{P}$, while $\pmb{x}_2$ can be computed using the MIP formulation~(\ref{mincv1}). Let $\pmb{x}'$ be the solution among $\pmb{x}_1$, $\pmb{x}_2$ that has smaller value of~(\ref{case0}) breaking ties arbitrarily. Assume that $\pmb{x}^*\in \mathcal{X}$ is an optimal solution to $\textsc{Distr}~\mathcal{P}$. Then
$$\min\left\{\pmb{b}^T\pmb{x}^*, {\rm CVaR}^{\alpha}_{\widehat{\mathbb{P}}^{N}}[\randxiT\pmb{x}^*]+\frac{1}{l}N\epsilon\right\}\geq \min\left\{\pmb{b}^T\pmb{x}', {\rm CVaR}^{\alpha}_{\widehat{\mathbb{P}}^{N}}[\randxiT\pmb{x}']+\frac{1}{l}N\epsilon\right\}$$
and, according to Proposition~\ref{propab}, $\pmb{x}'$ is also optimal to $\textsc{Distr}~\mathcal{P}$. The algorithm for computing $\pmb{x}'$ can be faster than solving the MIP formulation proposed in~\cite{WYSZ19}.

\subsection{A general convex support set $\Xi$ and $q\in \overline{\Rset}_{\geq 1}$}
\label{sec4_3}

In this section, we will propose some methods of solving $\textsc{Distr}~\mathcal{P}$ for any bounded, closed, and convex support $\Xi$ and $q\in \overline{\Rset}_{\geq 1}$. To simplify the presentation, we will also assume that $\alpha=\frac{l}{N}$ for $l\in [N]$.
Let us first rewrite~(\ref{pf}) as follows:

\begin{equation}
\label{rg0}
	\begin{array}{llll}
		\min & z\\
			& {\rm CVaR}^{\alpha}_{\mathbb{P}^N}[\randxiT\pmb{x}]\leq z, &\forall  \mathbb{P}^N \in \mathcal{B}_{\epsilon}(\widehat{\mathbb{P}}^N)\\
			& \pmb{x}\in \mathcal{X}.
	\end{array}
\end{equation}

\begin{prop}
\label{propmodg}
If $\alpha=\frac{l}{N}$ for $l\in [N]$, then for each $\mathbb{P}^N=(\pmb{\xi}_1,\dots,\pmb{\xi}_N)$
$$
{\rm CVaR}^{\alpha}_{\mathbb{P}^N}[\randxiT\pmb{x}]=\max_{\overset{A\subseteq \{1,\dots,N\}}{|A|=l}}\frac{1}{l}\left(\sum_{i\in A} \pmb{\xi}_i^T\right)\pmb{x}$$
\end{prop}
\begin{proof}
This is a direct consequence of Proposition~\ref{propowa}. If $\alpha=\frac{l}{N}$ for some $l\in [N]$, then the Conditional Value at Risk is the average of the $l$ largest costs of $\pmb{x}$ under $\mathbb{P}^N=(\pmb{\xi}_1,\dots,\pmb{\xi}_N)$.
\end{proof}
Using Proposition~\ref{propmodg}, we can represent~(\ref{rg0}) as
\begin{equation}
\label{rg1}
	\begin{array}{llll}
		\min & z\\
			& \displaystyle \frac{1}{l}\left(\sum_{i\in A} \pmb{\xi}_i^T\right)\pmb{x}\leq z, &\forall  (\pmb{\xi}_1,\dots,\pmb{\xi}_N) \in \mathcal{B}_{\epsilon}(\widehat{\mathbb{P}}^N), \forall A\subseteq \{1,\dots,n\}, |A|=l,\\
			& \pmb{x}\in \mathcal{X}.
	\end{array}
\end{equation}
Let us define the uncertainty set
 $$\mathcal{U}=\left\{\frac{1}{l}\sum_{i\in A}\pmb{\xi}_i:  (\pmb{\xi}_1,\dots,\pmb{\xi}_N)\in \mathcal{B}_{\epsilon}(\widehat{\mathbb{P}}^N), A\subseteq \{1,\dots,N\}, |A|=l\right\}.$$
 We can then represent~(\ref{rg1}) more briefly as follows:
\begin{equation}
\label{rg2}
	\begin{array}{llll}
		\min & z\\
			& \displaystyle  \pmb{\zeta}^T\pmb{x}\leq z, &\forall  \pmb{\zeta}\in \mathcal{U},\\
			& \pmb{x}\in \mathcal{X}.
	\end{array}
\end{equation}

Problem~(\ref{rg2}) can be solved by using a standard row generation algorithm shown in Figure~\ref{alg1}. For any finite subset $\mathcal{U}' \subseteq \mathcal{U}$~(\ref{rg2}) is a mixed integer program. By solving it, we get a solution $\pmb{x}^*\in \mathcal{X}$ with a lower bound $z^*_{LB}$ on the optimal objective value $z^*$ of~(\ref{rg2}). We then compute a worst probability distribution $\mathbb{P}^N=(\pmb{\xi}_1,\dots,\pmb{\xi}_N)\in \mathcal{B}_{\epsilon}(\widehat{\mathbb{P}}^N)$ for~$\pmb{x}^*$. Evaluating ${\rm CVaR}^{\alpha}_{\mathbb{P}^N}[\randxiT\pmb{x}^*]$ gives us an upper bound $z^*_{UB}$ on the optimal objective value of~(\ref{rg2}). The cut for $\pmb{x}^*$ can be generated by identifying the subset of $l$ worst realizations for $\pmb{x}^*$ in $\mathbb{P}^N$. Namely, $\pmb{\zeta}=\frac{1}{l}\sum_{i=1}^l \pmb{\xi}_{\sigma(i)}$, where $\pmb{\xi}_{\sigma(1)}^T\pmb{x}^*\geq \dots\geq \pmb{\xi}_{\sigma(n)}^T\pmb{x}^*$, is added to $\mathcal{U}'$. We have to initialize the set $\mathcal{U}'$ to start the algorithm. In the simplest approach, we can add one realization, say $\pmb{\zeta}=\frac{1}{l}\sum_{i=1}^l \hat{\pmb{\xi}}_i$ to $\mathcal{U}'$.

\begin{figure}[ht]
\begin{algorithmic}[1]
\STATE Generate a nonempty subset $\mathcal{U}'\subseteq \mathcal{U}$.
\STATE  $z^*_{UB}:=+\infty$
\REPEAT
\STATE Solve the problem~(\ref{rg2})  with  $\mathcal{U}'$ obtaining $\pmb{x}^*$ with $z^*_{LB}$
\STATE Compute a worst probability distribution $\mathbb{P}^N=(\pmb{\xi}_1,\dots,\pmb{\xi}_N)\in \mathcal{B}_{\epsilon}(\widehat{\mathbb{P}}^N)$ for $\pmb{x}^*$.
\IF{${\rm CVaR}^{\alpha}_{\mathbb{P}^N}[\randxiT\pmb{x}^*]<z^*_{UB}$}
	\STATE $\pmb{x}':=\pmb{x}^*$
\ENDIF
\STATE $z^*_{UB}:=\min\{z^*_{UB}, {\rm CVaR}^{\alpha}_{\mathbb{P}^N}[\randxiT\pmb{x}^*]\}$
\STATE $\mathcal{U}':=\mathcal{U}' \cup \{\frac{1}{l}\sum_{i=1}^l \pmb{\xi}_{\sigma(i)}\}$, where $\pmb{\xi}_{\sigma(1)}^T\pmb{x}^*\geq \dots\geq \pmb{\xi}_{\sigma(n)}^T\pmb{x}^*$
\UNTIL{$(z^*_{UB}-z^*_{LB})/z^*_{LB}\leq \varepsilon$}
\STATE \textbf{return} $\pmb{x}'$
\end{algorithmic}
\caption{A row generation algorithm.}\label{alg1}
\end{figure}

Unfortunately, computing a worst probability distribution for $\pmb{x}^*$ in step~5 of the algorithm is not a trivial task. Using Lemma~\ref{lemworst}, we will consider the following optimization problem:
 $$\max_{\mathbb{P}^N \in \mathcal{B}_{\epsilon}(\widehat{\mathbb{P}}^N)} {\rm CVaR}^{\alpha}_{\mathbb{P}^N}[\randxiT\pmb{x}], $$
 where the solution $\pmb{x}\in \mathcal{X}$ is fixed.
 If $\alpha=\frac{l}{N}$ for some $l\in [N]$, then, according to Proposition~\ref{propowa}, ${\rm CVaR}^{\alpha}_{\mathbb{P}^N}[\randxiT\pmb{x}]$ is the average of the $l$ largest costs under $\mathbb{P}^N$. Therefore, the problem can be expressed as follows:
 \begin{equation}
 \label{f0}
 	\begin{array}{lll}
		\max  & \displaystyle \frac{1}{l}\sum_{i=1}^N  y_i\pmb{\xi}_i^T\pmb{x}\\
		& \displaystyle \sum_{i=1}^N ||\pmb{\xi}_i- \hat{\pmb{\xi}}_i||_q\leq N\epsilon,\\
		& \displaystyle \sum_{i=1}^N y_i=l, \\
		& \pmb{\xi}_i\in \Xi, &  i\in [N],\\
		& y_i\in \{0,1\}, & i\in [N].
	\end{array}
\end{equation}
 The nonlinear terms $y_i\pmb{\xi}_i^T\pmb{x}$ in~(\ref{f0}) can be linearized using standard techniques. Problem~(\ref{f0}) can be solved by using some available solvers, for example, CPLEX~\cite{CPLEX}, when $q=1,2,\infty$ and $\Xi$ is described by linear or second-order cone constraints. Observe that~(\ref{f0}) becomes a tractable convex problem when $l=N$ or $l=1$. In the former case, the CVaR criterion is the expectation, while in the latter one, it is the maximum cost of $\pmb{x}$ under $\mathbb{P}^N$.    In general, (\ref{f0}) is tractable when $N$ is not large, which is true in many practical applications when the available sample size is small.

\subsubsection{Approximation algorithm}
\label{sec4_4}

 In this section, we propose a method for computing approximate solutions to $\textsc{Distr}~\mathcal{P}$ with a convex, closed, and bounded support set $\Xi$. Let $\overline{\pmb{\xi}}$ maximize $\pmb{1}^T\pmb{\xi}$ over $\pmb{\xi}\in \Xi$, where $\pmb{1}=(1,1,\dots,1)^T$. Suppose that $\alpha\in (\frac{l-1}N, \frac{l}{N}]$ for some $l\in [N]$.
Let $c\geq 1$ be the smallest constant such that
\begin{equation}
\label{defc}
||\overline{\pmb{\xi}}-\hat{\pmb{\xi}}_i||_q\leq c \frac{\epsilon N}{l},  \;\; i\in [N].
\end{equation}
Define $\mathbb{P}^{'N}=(\pmb{\xi}_1',\dots,\pmb{\xi}_N')$, where
 $$\pmb{\xi}'_i=\hat{\pmb{\xi}}_i+\frac{\overline{\pmb{\xi}}-\hat{\pmb{\xi}}_i}{c}, \; i\in [N].$$
  Since $\hat{\pmb{\xi}}_i\in \Xi$ and $\overline{\pmb{\xi}}\in \Xi$, we get $\pmb{\xi}_i'\in \Xi$. Indeed, $\pmb{\xi}_i'=\frac{c-1}{c}\hat{\pmb{\xi}}_i+\frac{1}{c}\overline{\pmb{\xi}}$, which is a convex combination of $\hat{\pmb{\xi}}_i'$ and $\overline{\pmb{\xi}}$ for each $c\geq 1$. Vector $\pmb{\xi}_i'$ lies on the line segment joining $\hat{\pmb{\xi}}_i$ and $\overline{\pmb{\xi}}$. The distribution $\mathbb{P}^{'N}$ can be seen as a distorted sample $\widehat{\mathbb{P}}^N$ in which the realizations $\hat{\pmb{\xi}}_i$ are moved towards a bad realization $\overline{\pmb{\xi}}$.
   Let $\pmb{\zeta}=(\zeta_1,\dots,\zeta_n)\in \Rset^n_{+}$, where $\zeta_i$ maximizes the component $\xi_i$ over $\pmb{\xi}\in \Xi$. An example is shown in Figure~\ref{fig4}.
 
  \begin{figure}[ht]
\centering
\includegraphics[height=6cm]{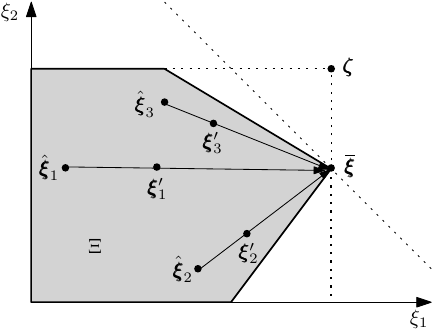}
\caption{A sample support set $\Xi\subseteq \mathbb{R}^2$, distribution $\mathbb{P}^{'N}=(\pmb{\xi}_1',\pmb{\xi}_2',\pmb{\xi}_3')$, $\overline{\pmb{\zeta}}$.}\label{fig4}
\end{figure}
 
 \begin{thm}
	If  $\pmb{x}'\in \mathcal{X}$ minimizes ${\rm CVaR}^{\alpha}_{\mathbb{P}'}[\randxiT\pmb{x}]$ and $\pmb{\zeta}^{T}\pmb{x}'\leq b\cdot \overline{\pmb{\xi}}^T\pmb{x}'$ for some constant $b\geq 1$. Then $\pmb{x}'$ is a $(b\cdot c)$-approximate solution to $\textsc{Distr}~\mathcal{P}$.
\end{thm}
\begin{proof}

  Let $\pmb{x}^*$ be an optimal solution to $\textsc{Distr}~\mathcal{P}$ with a worst  probability distribution $\mathbb{P}^{*N}\in \mathcal{B}_{\epsilon}(\widehat{\mathbb{P}}^N) $. We will show that
\begin{equation}
\label{z00}
	{\rm CVaR}^{\alpha}_{\mathbb{P}^{*N}}[\randxiT\pmb{x}^*]\geq {\rm CVaR}^{\alpha}_{\mathbb{P}^{'N}}[\randxiT\pmb{x}^*]\geq {\rm CVaR}^{\alpha}_{\mathbb{P}^{'N}}[\randxiT\pmb{x}'].
\end{equation}
The second inequality in~(\ref{z00}) follows directly from the assumption of the theorem. It remains to show the first one.
Let us represent $\mathbb{P}^{'N}$ as follows
$$\mathbb{P}^{'N}=(\hat{\pmb{\xi}}_{1}+\pmb{\Delta}'_{1},\dots,\hat{\pmb{\xi}}_{N}+\pmb{\Delta}'_N),$$
where $\pmb{\Delta}'_i\in \Rset^n$ for $i\in [N]$.
Using Proposition~\ref{propowa}, we get
$${\rm CVaR}^{\alpha}_{\mathbb{P}^{'N}}[\randxiT\pmb{x}^*]=\sum_{i=1}^l w_i(\hat{\pmb{\xi}}_{\sigma(i)}+\pmb{\Delta}'_{\sigma(i)})^T\pmb{x}^*,$$
where $(\hat{\pmb{\xi}}_{\sigma(1)}+\pmb{\Delta}'_{\sigma(1)})^T\pmb{x}^*\geq \dots \geq (\hat{\pmb{\xi}}_{\sigma(l)}+\pmb{\Delta}'_{\sigma(l)})^T\pmb{x}^*\geq \dots\geq (\hat{\pmb{\xi}}_{\sigma(N)}+\pmb{\Delta}'_{\sigma(N)})^T\pmb{x}^* $ and $\sigma$ is permutation of $[N]$.
Let us construct a probability distribution $\mathbb{P}^{''N}$ in which $\pmb{\xi}^{''}_i=\hat{\pmb{\xi}}_i+\pmb{\Delta}'_i=\hat{\pmb{\xi}}_i+\frac{1}{c}(\overline{\pmb{\xi}}-\hat{\pmb{\xi}}_i)$
 if $i\in \{\sigma(1),\dots,\sigma(l)\}$ and $\pmb{\xi}^{''}_i=\hat{\pmb{\xi}}_i$, otherwise. Obviously
\begin{equation}
\label{z0}
{\rm CVaR}^{\alpha}_{\mathbb{P}^{'N}}[\pmb{\xi}^T\pmb{x}^*]={\rm CVaR}^{\alpha}_{\mathbb{P}^{''N}}[\pmb{\xi}^T\pmb{x}^*].
\end{equation}
and
$$\frac{1}{N}\sum_{i=1}^N||\pmb{\xi}''_i-\hat{\pmb{\xi}}_i||_q= 
\frac{1}{N}\sum_{i=1}^l \left \lVert\frac{\overline{\pmb{\xi}}-\hat{\pmb{\xi}}_{\sigma(i)}}{c}\right \rVert_q=\frac{1}{N}\sum_{i=1}^l \frac{1}{c}||\overline{\pmb{\xi}}-\hat{\pmb{\xi}}_{\sigma(i)}||_q\leq \frac{1}{N}\sum_{i=1}^l\frac{\epsilon N}{l}= \epsilon,$$
where the last inequality follows from~(\ref{defc}).
Hence, we conclude that $\mathbb{P}^{''N}\in \mathcal{B}_{\epsilon}(\widehat{\mathbb{P}}^N)$. Consequently, 
${\rm CVaR}^{\alpha}_{\mathbb{P}^{''N}}[\randxiT\pmb{x}^*]\leq  {\rm CVaR}^{\alpha}_{\mathbb{P}^{*N}}[\randxiT\pmb{x}^*]$, which together with~(\ref{z0}) imply~(\ref{z00}).
Let $$\mathbb{P}_u^N=\left(\frac{\overline{\pmb{\xi}}}{c},\dots,\frac{\overline{\pmb{\xi}}}{c}\right).$$
Because $\hat{\pmb{\xi}}_i\in \Rset^n_+$, $c\geq 1$, and thus $(\hat{\pmb{\xi}}_i+\frac{\overline{\pmb{\xi}}-\hat{\pmb{\xi}}_i}{c})^T\pmb{x}'=((1-\frac{1}{c})\hat{\pmb{\xi}}_i+\frac{1}{c}\overline{\pmb{\xi}})^T\pmb{x}'\geq \frac{1}{c}\overline{\pmb{\xi}}^{T}\pmb{x}'$ for each $i\in [N]$, we obtain
 \begin{equation}
 \label{z01}
 {\rm CVaR}^\alpha_{\mathbb{P}^{'N}}[\randxiT\pmb{x}']\geq  {\rm CVaR}^\alpha_{\mathbb{P}_u^{N}}[\randxiT\pmb{x}']=\frac{1}{c}\overline{\pmb{\xi}}^T\pmb{x}'\geq \frac{1}{cb}\pmb{\zeta}^T\pmb{x}'\geq
  \frac{1}{bc}{\rm CVaR}^\alpha_{\mathbb{P}^{N}}[\randxiT\pmb{x}']
 \end{equation}
 for each $\mathbb{P}^N\in\mathcal{B}_{\epsilon}(\widehat{\mathbb{P}}^N)$. Inequalities~(\ref{z01}) and~(\ref{z00}) imply the theorem.
\end{proof}

For the support set $\Xi_{\pmb{a},\pmb{b}}$, we get $\pmb{\zeta}=\overline{\pmb{\xi}}$, which leads to the following result.
 \begin{thm}
	If the support set is $\pmb{\Xi}_{\pmb{a},\pmb{b}}$ and $\pmb{x}'$ minimizes ${\rm CVaR}^{\alpha}_{\mathbb{P}'}[\randxiT\pmb{x}]$, then $\pmb{x}'$ is a $c$-approximate solution to $\textsc{Distr}~\mathcal{P}$.
\end{thm}

The algorithm described in this section can be seen as another method of dealing with the sample $\widehat{\mathbb{P}}^N$. If $N$ is small, it may be advantageous to distort this sample towards some bad realization $\overline{\pmb{\xi}}$ to obtain more robust solutions. Observe that we do not need to know the whole shape of the support set $\Xi$. It is only enough to assume that $\Xi$ is convex and $\overline{\pmb{\xi}}\in \Xi$. One can also try other methods of choosing $\overline{\pmb{\xi}}$. For example $\overline{\pmb{\xi}}$ can minimize a distance to $\pmb{\zeta}$ over all $\pmb{\xi}\in \Xi$. We can also choose an arbitrary value of $c\geq 1$. We then get a heuristic solution whose performance can be compared to other approaches.

\section{Computational tests}
\label{sec5}

In this section, we will show the results of some computational tests. We will investigate the performance of the Wasserstein robust approach for the minimum \textsc{ 0-1 Knapsack} problem, which is a fundamental combinatorial optimization problem often used to test the proposed concepts (see, e.g.~\cite{BS04}). Recall that the deterministic \textsc{ 0-1 Knapsack} is already NP-hard.
 
 We are given a set of items $[n]$. For each item $i\in [n]$ a nonnegative weight $w_i$ is specified and the set of feasible solution is $\mathcal{X}=\{\pmb{x}\in \{0,1\}^n: \sum_{i\in [n]} w_i x_i \geq W\}$, where $W\geq 0$ is a given capacity. In our experiments, each $w_i$ was a random number drawn from the uniform probability distribution in the interval $[0,1]$ (i.e. $w_i\sim \mathcal{U}([0,1])$). The value of $W$ was then set to $0.4\sum_{i\in [n]} w_i$.  The random cost vector $\tilde{\pmb{\xi}}=(\tilde{\xi}_1,\dots,\tilde{\xi}_n)$ of the items was defined as follows. For each component $\tilde{\xi}_i$ of $\tilde{\pmb{\xi}}$ an interval $[\underline{\xi}_i, \overline{\xi}_i]$ was determined, where $\underline{\xi}_i=\max\{0, w_i-\alpha_i\}$,  $\alpha_i\sim \mathcal{U}([0,1])$ and $\overline{\xi}_i=\underline{\xi}_i+2\beta_i$, $\beta_i\sim \mathcal{U}([0,1])$. Then the component $\tilde{\xi}_i$ has a truncated normal distribution in the interval $[\underline{\xi}_i, \overline{\xi}_i]$ with the expected value $w_i$ and the standard deviation $0.7w_i$. All the random variables $\alpha_i, \beta_i$, $i\in [n]$, were mutually independent. The method of creating the instance can be justified by the fact that in practical applications, the items having larger weights also have larger costs. Observe that the support set $\Xi$ is the Cartesian product of $[\underline{\xi}_i, \overline{\xi}_i]$ for $i\in [n]$.
\begin{figure}[h!t]
\includegraphics[height=5.7cm]{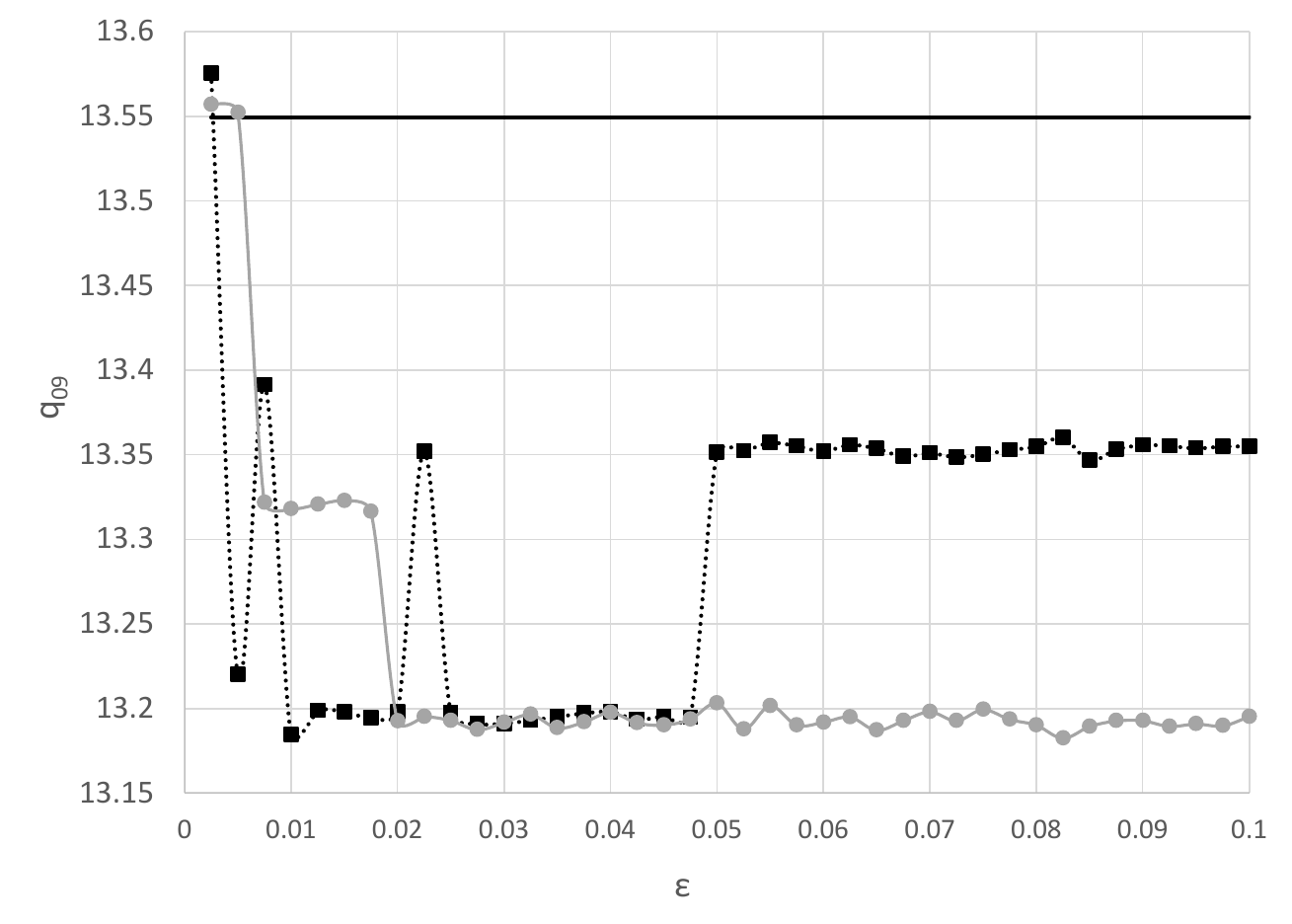}
\includegraphics[height=5.7cm]{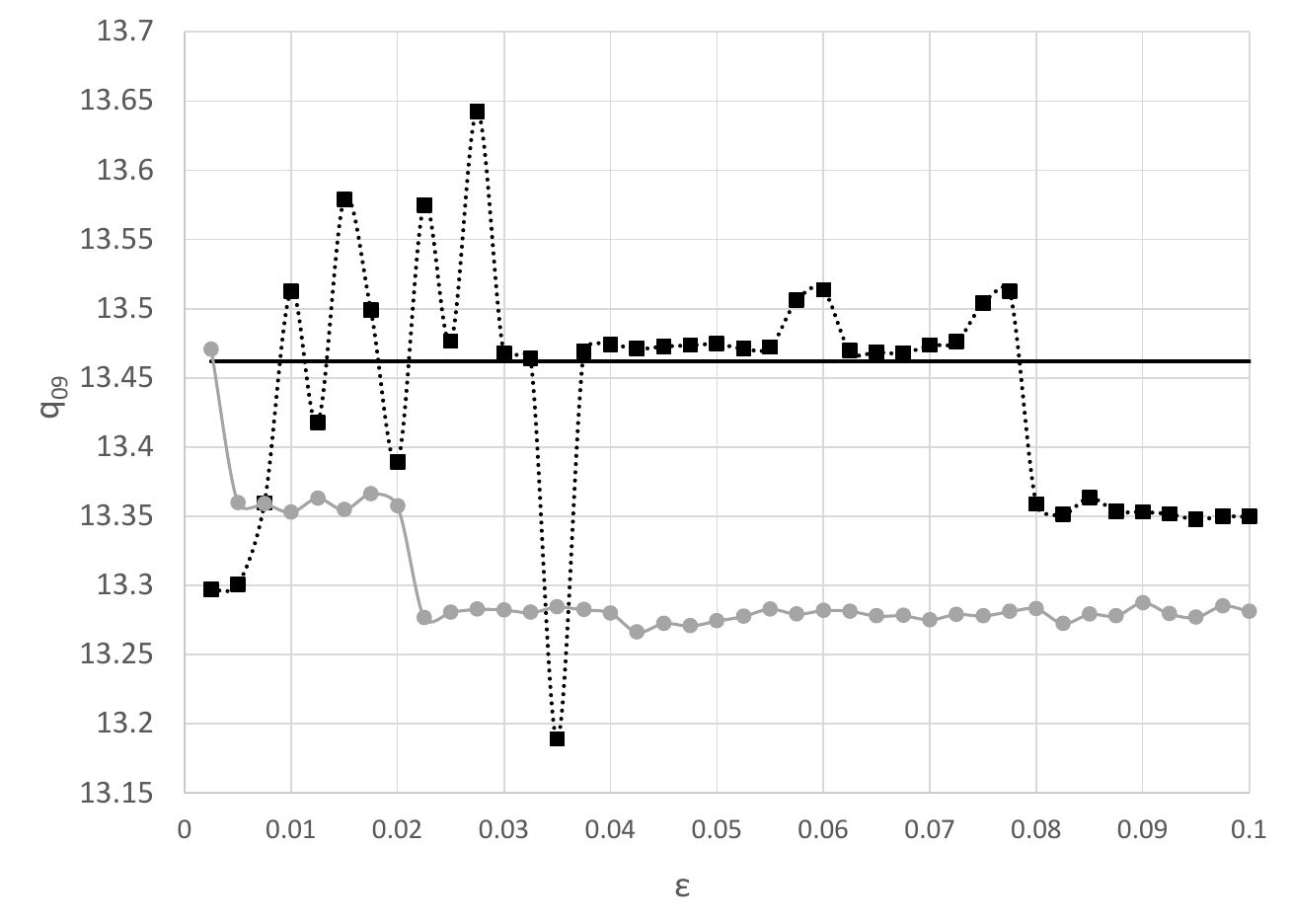}
\includegraphics[height=5.7cm]{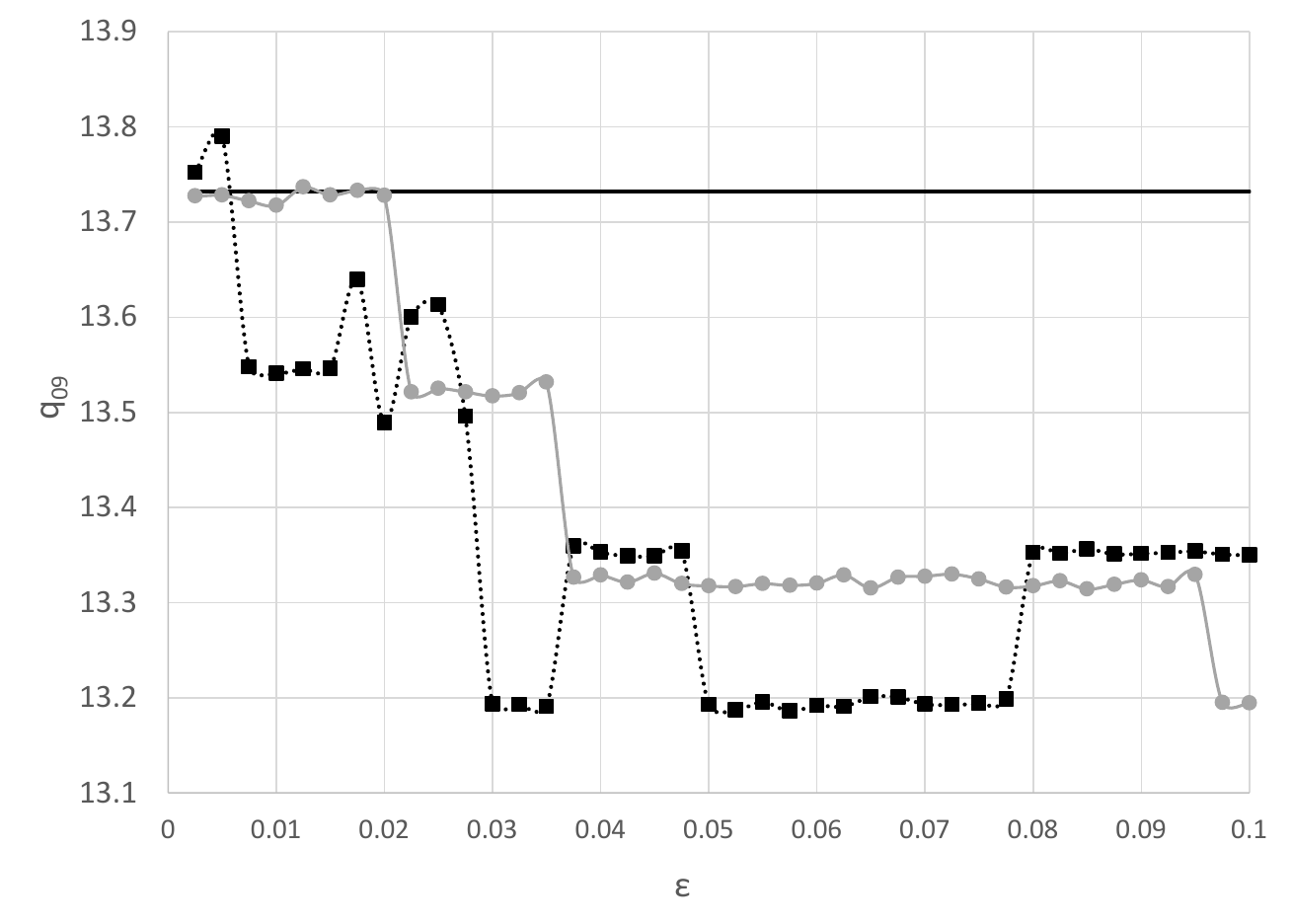}
\includegraphics[height=5.7cm]{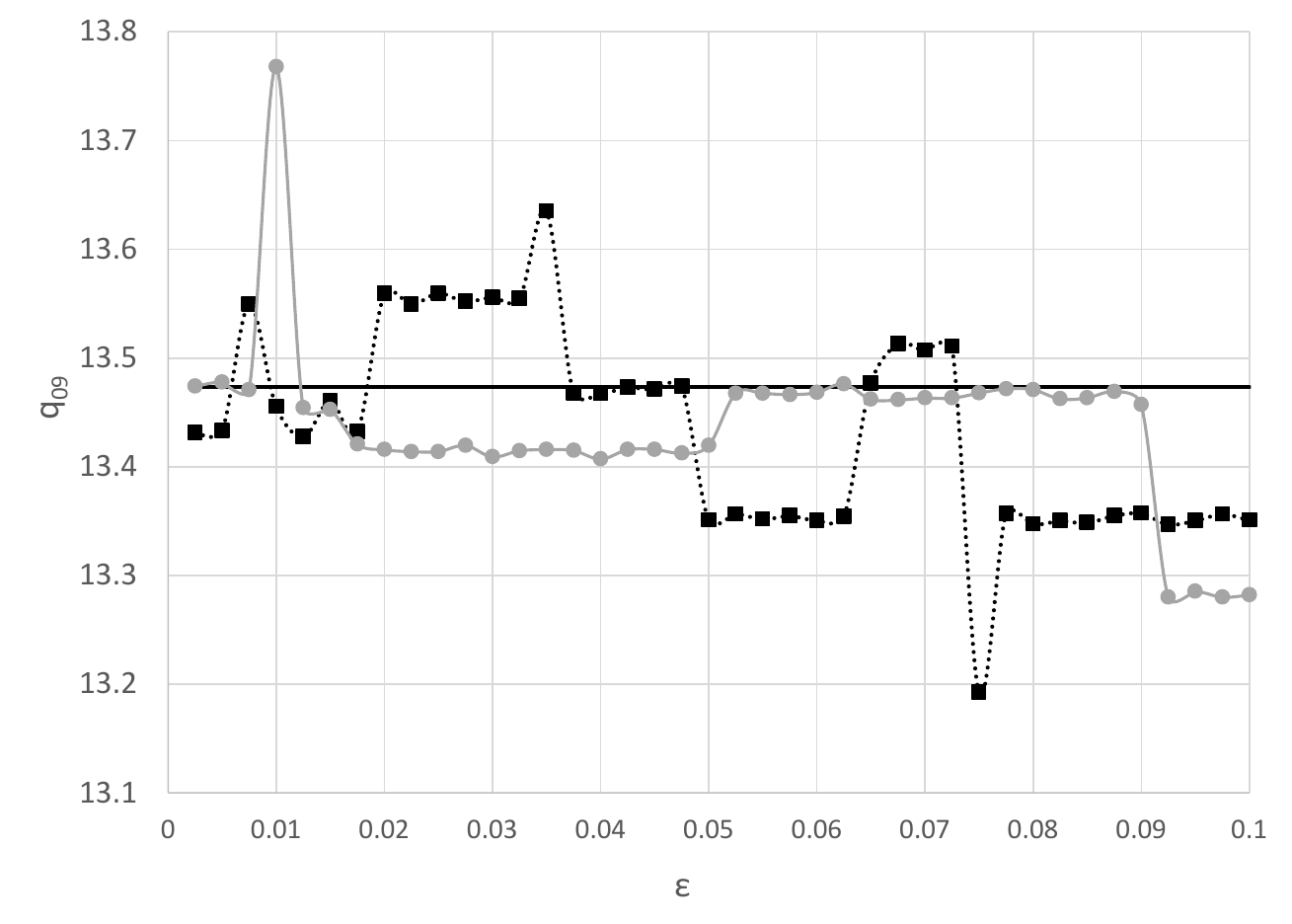}
\caption{The $0.9$-quantiles of $\pmb{x}_i$ (horizontal line), $\pmb{y}_i^{\epsilon}$ (dotted line) and $\pmb{z}_i^{\epsilon}$ (grey line) for four empirical distributions $\widehat{\mathbb{P}}^{N}_1, \dots,  \widehat{\mathbb{P}}^{N}_{4}$ of size~30. }\label{figexp1}
\end{figure}

\begin{figure}[h!t]
\centering
\includegraphics[height=8cm]{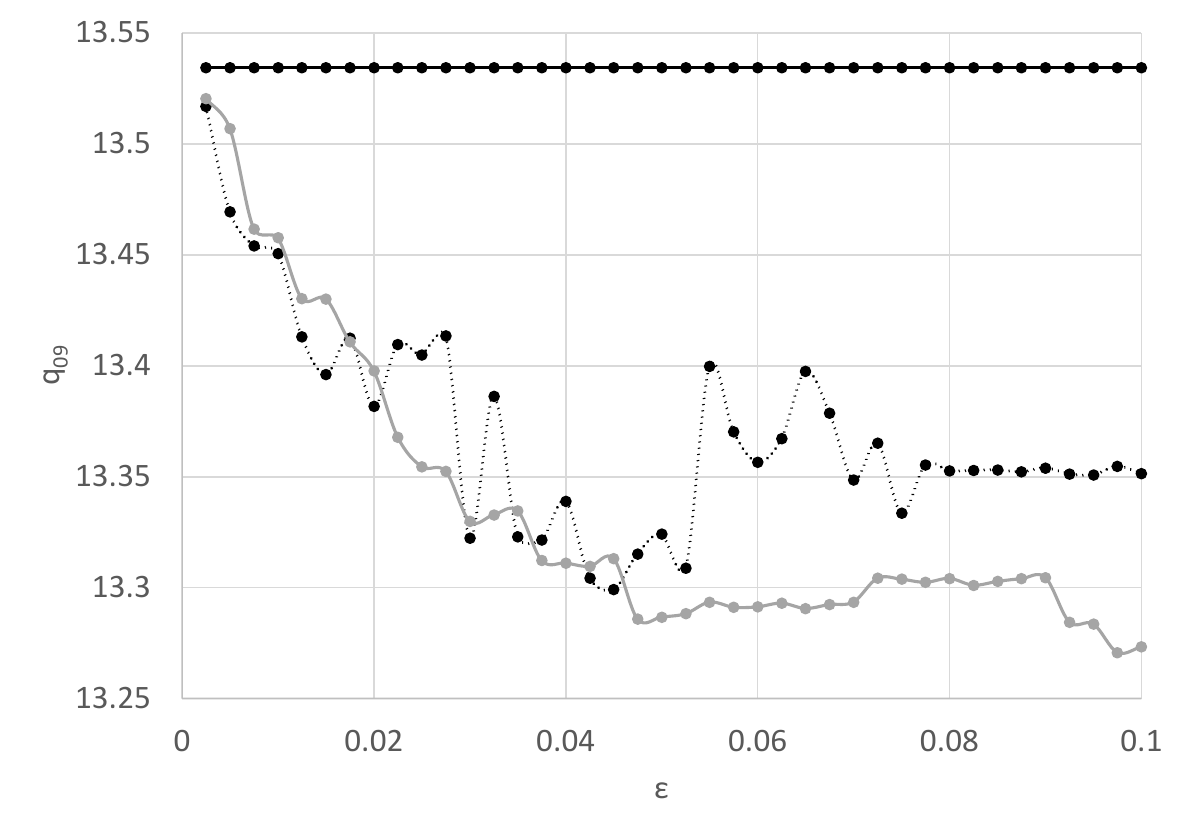}
\caption{The average $0.9$-quantiles of $\pmb{x}_i$ (horizontal line), $\pmb{y}_i^{\epsilon}$ (dotted line) and $\pmb{z}_i^{\epsilon}$ (grey line) for ten empirical distributions $\widehat{\mathbb{P}}^{N}_1, \dots,  \widehat{\mathbb{P}}^{N}_{10}$ of size~30.}
\end{figure}

The first experiment was carried out as follows. We fixed $n=100$ and generated an instance $(\pmb{w},W,\tilde{\pmb{\xi}})$ of the problem. We then generated 10 random samples of size $N=30$ of $\pmb{\tilde{\xi}}$ obtaining ten empirical distributions $\widehat{\mathbb{P}}^{N}_1, \dots,  \widehat{\mathbb{P}}^{N}_{10}$.  For each empirical distribution $\widehat{\mathbb{P}}^{N}_i$, $i=1,\dots,10$, we computed an optimal solution $\pmb{x}_i$ to $\textsc{CVaR}~\mathcal{P}$ with $\alpha=0.1$. We have thus applied the traditional approach based on the sample average approximation. Then for each $\epsilon=0.0025, 0.005, 0.00525,\dots,0.1$ and $i=1,\dots,10$, we computed an optimal solution $\pmb{y}_i^{\epsilon}$ to $\text{Distr}~\mathcal{P}$ with $\mathcal{B}_{\epsilon}(\widehat{\mathbb{P}}_i^N))$  using the row generation algorithm and a solution $\pmb{z}_i^{\epsilon}$ using the approximation algorithm proposed in Section~\ref{sec4_4}. We also chose $\alpha=0.1$ and $p=\infty$ in the definition of the Wasserstein ball. We now face the problem of comparing the solutions $\pmb{x}_i$ to $\pmb{y}_i^{\epsilon}$ and $\pmb{z}_i^{\epsilon}$.  The cost $\randxiT\pmb{x}$ of a given solution $\pmb{x}$ is a random variable. The robustness of $\pmb{x}$ can be measured by computing the $0.9$-quantile $q_{0.9}$ of  $\randxiT\pmb{x}$ - with probability $0.9$, the largest cost of $\pmb{x}$ will not exceed $q_{0.9}$. The solutions with smaller values of $q_{0.9}$ will be regarded as more robust. For a given solution, the quantile $q_{0.9}$ can be approximated by generating a very large sample of $\tilde{\pmb{\xi}}$. In our experiments, we chose a sample size of $100~000$ for this purpose.

Figure~\ref{figexp1} shows four representative examples among the ten generated. The horizontal line represents the $q_{0.9}$ quantile of $\pmb{x}_i$. The dotted line is the $q_{0.9}$ quantile of $\pmb{y}_i^{\epsilon}$ and the grey line is is the $q_{0.9}$ quantile of $\pmb{z}_i^{\epsilon}$ for each tested $\epsilon$. The performance of the solutions depends on the sample. In almost all cases, the performance of $\pmb{z}_i^{\epsilon}$, computed by the approximation algorithm, was better than $\pmb{x}_i$. For the solution $\pmb{y}_i^{\epsilon}$, the conclusion is not so decisive. In most cases, we also get better solutions, but sometimes the performance of $\pmb{z}_i^{\epsilon}$ was worse than $\pmb{x}_i$, especially when the solution $\pmb{x}_i$ has a smaller value of $q_{0.9}$. Figure~\ref{figexp1} shows the average performance (average values of $q_{0.9}$ of solutions) over all~10 samples.

In the second experiment, we focused on the approximation algorithm constructed in Section~\ref{sec4_4}. We again fixed $n=100$ and generated an instance $(\pmb{w},W,\tilde{\pmb{\xi}})$ of the problem. We then generated 10 random samples of size $N=30$ of $\pmb{\tilde{\xi}}$ obtaining the empirical distributions $\widehat{\mathbb{P}}^{N}_1, \dots,  \widehat{\mathbb{P}}^{N}_{10}$. We fixed $\alpha=0.5$. For this choice of $\alpha$, the value of $l$ in the model~(\ref{f0}) is equal to~15, and the model turned out to be hard to solve. Therefore, we only applied the approximation algorithm for the tested instances. For each empirical distribution $\widehat{\mathbb{P}}^{N}_i$, $i=1,\dots,10$, we computed an optimal solution $\pmb{x}_i$ to $\textsc{CVaR}~\mathcal{P}$.  Then for each $\epsilon=0.0025, 0.005, 0.00525,\dots,1$ and $i=1,\dots,10$, we computed  a solution $\pmb{z}_i^{\epsilon}$ using the approximation algorithm proposed in Section~\ref{sec4_4}. We choose $p=\infty$ in the definition of the Wasserstein ball. As in the previous experiment, the robustness of $\pmb{x}_i$ and $\pmb{z}_i^{\epsilon}$ were evaluated by computing the $q_{0.9}$ quantiles for each solution, using a simulation on a very large sample. The results for four representative cases are shown in Figure~\ref{figexp3}. In the majority of cases, the performance of $\pmb{z}_i^{\epsilon}$ was better than $\pmb{x}_i$, except for the one case when the performance of $\pmb{x}_i$ itself was good and the performance of $\pmb{z}_i^{\epsilon}$ was similar. Interestingly, in all the cases, the value of $\epsilon$ should not be too small and too large. For $\epsilon\geq 0.9$ a large deterioration in the performance of $\pmb{z}_i^{\epsilon}$ was observed. The average performance of the approximation algorithm over ten random samples is shown in Figure~\ref{figexp4}. On average, the best choice of $\epsilon$ for the tested instances was between $0.1$ and $0.7$.

\begin{figure}[h!t]
\includegraphics[height=5.7cm]{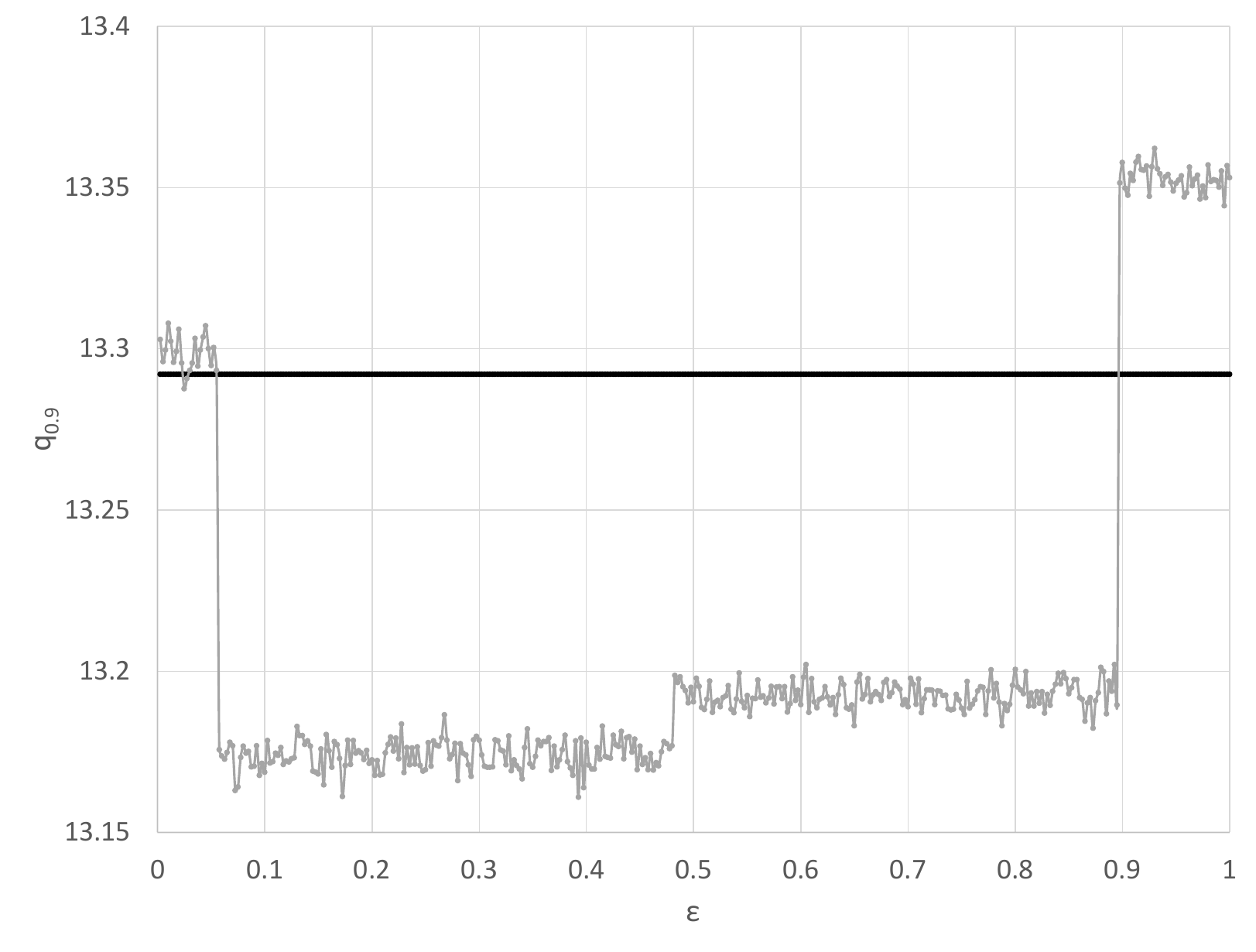}
\includegraphics[height=5.7cm]{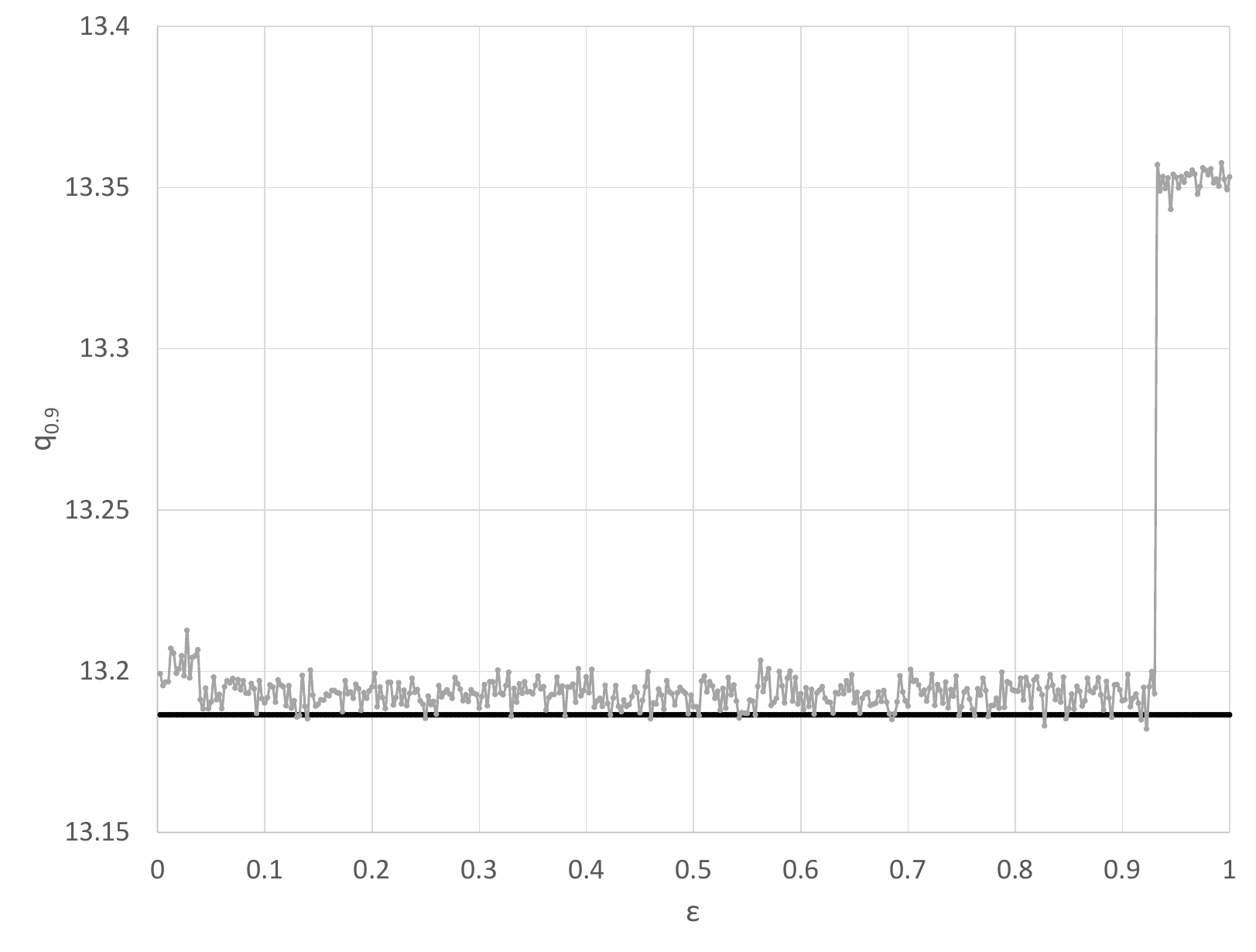}
\includegraphics[height=5.7cm]{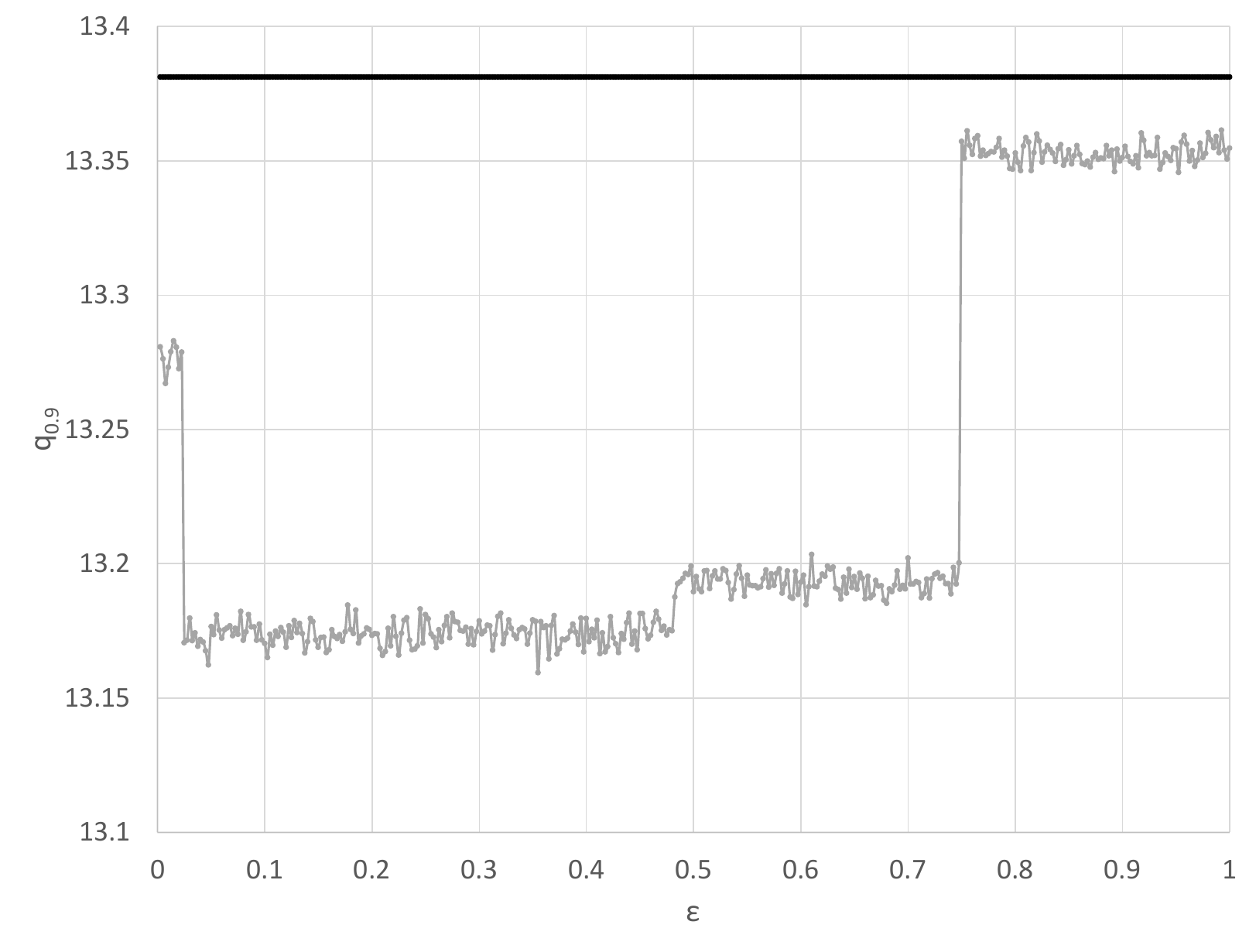}
\includegraphics[height=5.7cm]{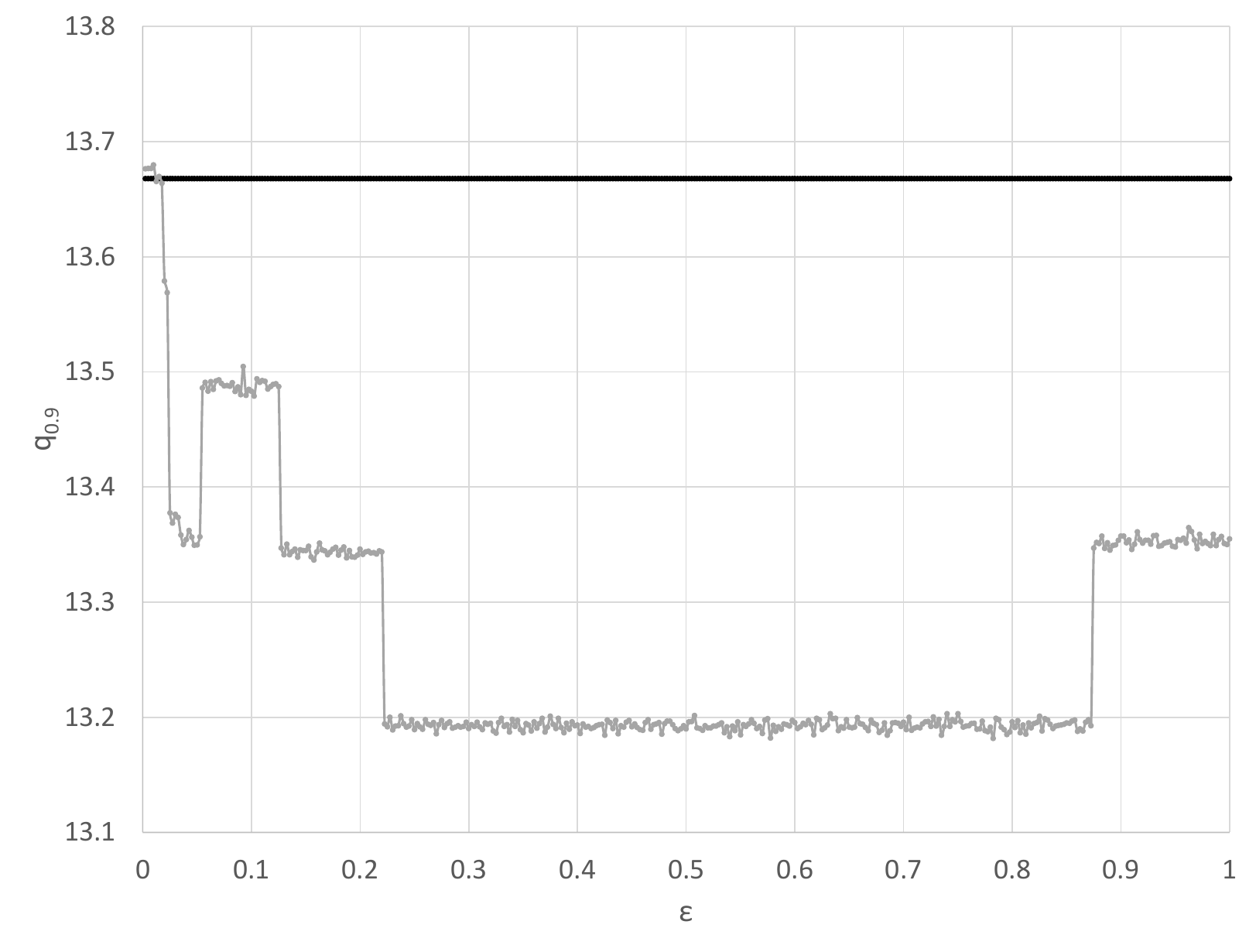}
\caption{The $0.9$-quantiles of $\pmb{x}_i$ (horizontal line) and $\pmb{z}_i^{\epsilon}$ (grey line) for four empirical distributions $\widehat{\mathbb{P}}^{N}_1, \dots,  \widehat{\mathbb{P}}^{N}_{4}$ of size~30. }\label{figexp3}
\end{figure}

\begin{figure}[h!t]
\centering
\includegraphics[height=8cm]{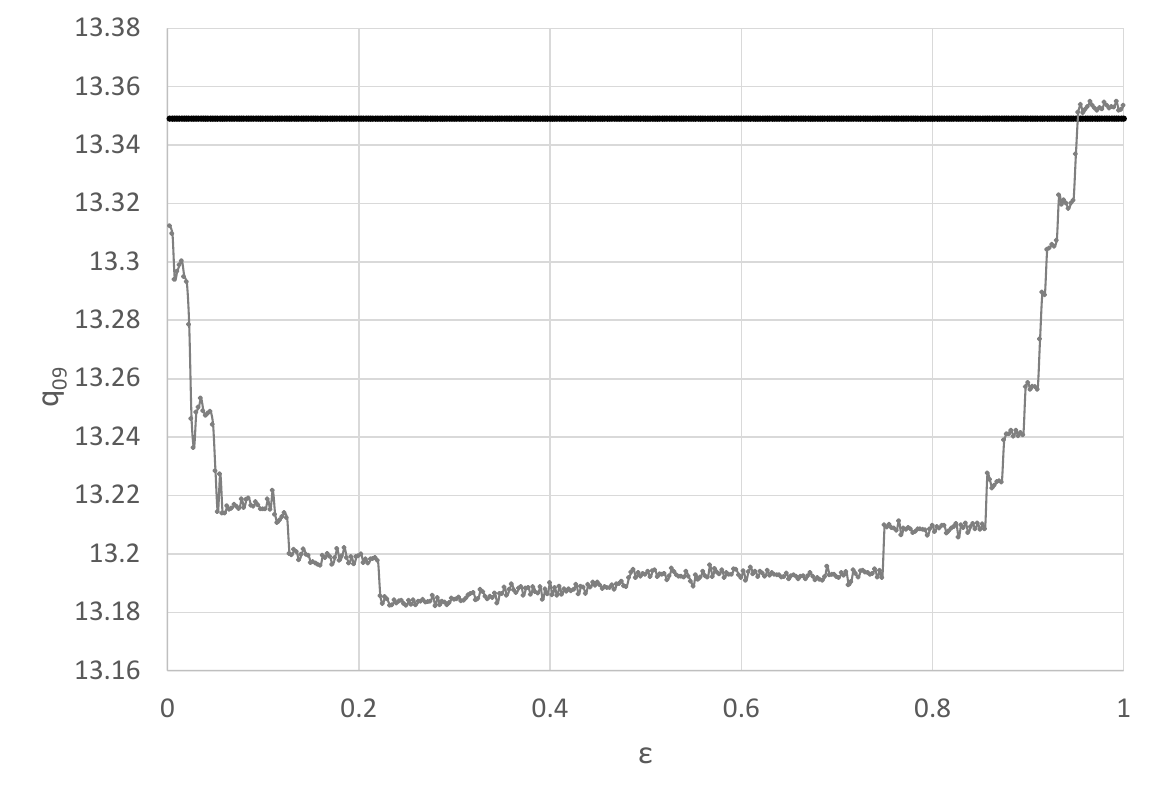}
\caption{The average $0.9$-quantiles of $\pmb{x}_i$ (horizontal line) and $\pmb{z}_i^{\epsilon}$ (grey line) for ten empirical distributions $\widehat{\mathbb{P}}^{N}_1, \dots,  \widehat{\mathbb{P}}^{N}_{10}$ of size~30.} \label{figexp4}
\end{figure}

\section{Conclusions}

In this paper, we discussed the class of combinatorial optimization problems with uncertain costs in the objective function. Assuming that a sample of cost realizations is available, we used the Wasserstein distance to define a set of admissible probability distributions for the random costs vector. We then seek a solution minimizing the CVaR risk measure for the worst probability distribution. We showed that the problem is strongly NP-hard for basic combinatorial optimization problems (in particular for the shortest path problem), except for the case when CVaR becomes the expectation. We showed some exact and approximation methods of solving the problem. In particular, we proposed a row generation algorithm which can be applied to arbitrary convex support sets. We also constructed an approximation algorithm whose idea is to minimize CVaR for an appropriately distorted sample.  Some computational tests indicated that the distributionally robust approach, based on the Wasserstein distance, may give better results than the sample average approximation.

An interesting open problem is the complexity of computing a worst probability distribution for a given solution. In this paper, we used a mixed integer programming formulation for this purpose, which, however, can be complex to solve for larger instances. More efficient algorithms for particular support sets (for example, ellipsoidal ones) should be constructed. Also, the proposed algorithms should be tested on other combinatorial problems, possibly considering real data.

\subsubsection*{Acknowledgements}
The authors were supported by
 the National Science Centre, Poland, grant 2022/45/B/HS4/00355.


\begin{thebibliography}{10}

\bibitem{AMO93}
R.~K. Ahuja, T.~L. Magnanti, and J.~B. Orlin.
\newblock {\em Network {F}lows: theory, algorithms, and applications}.
\newblock Prentice Hall, Englewood Cliffs, New Jersey, 1993.

\bibitem{BN09}
A.~Ben-Tal, L.~El~Ghaoui, and A.~Nemirovski.
\newblock {\em Robust optimization}.
\newblock Princeton Series in Applied Mathematics. Princeton University Press,
  Princeton, NJ, 2009.

\bibitem{BS04}
D.~Bertsimas and M.~Sim.
\newblock The price of robustness.
\newblock {\em Operations research}, 52:35--53, 2004.

\bibitem{BV04}
S.~P. Boyd and L.~Vandenberghe.
\newblock {\em Convex optimization}.
\newblock Cambridge university press, 2004.

\bibitem{DW13}
V.~G. Deineko and G.~J. Woeginger.
\newblock Complexity and in-approximability of a selection problem in robust
  optimization.
\newblock {\em 4OR - A Quarterly Journal of Operations Research}, 11:249--252,
  2013.

\bibitem{DY10}
E.~Delage and Y.~Ye.
\newblock Distributionally robust optimization under moment uncertainty with
  application to data-deriven problems.
\newblock {\em Operations Research}, 58:595--612, 2010.

\bibitem{DK12}
A.~Dolgui and S.~Kovalev.
\newblock Min-max and min-max (relative) regret approaches to representatives
  selection problem.
\newblock {\em 4OR - A Quarterly Journal of Operations Research}, 10:181--192,
  2012.

\bibitem{KE18}
P.~M. Esfahani and D.~Kuhn.
\newblock Data-driven distributionally robust optimization using the
  {W}asserstein metric: performance guarantees and tractable reformulations.
\newblock {\em Mathematical Programming, Ser. A}, 171:115--166, 2018.

\bibitem{FG14}
N.~Fournier and A.~Guillin.
\newblock On the rate of convergence in {W}asserstein distance of the empirical
  measure.
\newblock {\em Probability Theory and Related Fields}, 162:1--32, 2014.

\bibitem{GK23}
R.~Gao and A.~J. Kleywegt.
\newblock Distributionally robust stochastic optimization with {W}asserstein
  distance.
\newblock {\em Mathematics of Operations Research}, 48:603--1211, 2023.

\bibitem{H92}
R.~Hassin.
\newblock Approximation {S}chemes for the {R}estricted {S}hortest {P}ath
  {P}roblem.
\newblock {\em Mathematics of Operations Research}, 17:36--42, 1992.

\bibitem{CPLEX}
{IBM~ILOG~CPLEX~Optimization~Studio}.
\newblock {CPLEX} {U}ser's manual.
\newblock {https://www.ibm.com}.

\bibitem{JKZ23}
M.~Jackiewicz, A.~Kasperski, and P.~Zieli{\'{n}}ski.
\newblock Solving {W}asserstein distributionally robust combinatorial
  optimization problems.
\newblock In {\em Proceedings of OR 2023, to appear}, 2023.

\bibitem{KW94}
P.~Kall and S.~W. Wallace.
\newblock {\em Stochastic Programming}.
\newblock John Wiley and Sons, 1994.

\bibitem{KR58}
L.~V. Kantorovich and G.~S. Rubinshtein.
\newblock On a space of totally additive functions.
\newblock {\em Vestnik, Leningrad University}, 13:52--59, 1958.

\bibitem{KKZ15}
A.~Kasperski, A.~Kurpisz, and P.~Zieli{\'n}ski.
\newblock Approximability of the robust representatives selection problem.
\newblock {\em Operations Research Letters}, 43:16--19, 2015.

\bibitem{KZ15}
A.~Kasperski and P.~Zieli{\'n}ski.
\newblock Combinatorial optimization problems with uncertain costs and the
  {OWA} criterion.
\newblock {\em Theoretical Computer Science}, 565:102--112, 2015.

\bibitem{KWFH22}
H.~Kim, D.~Watel, A.~Faye, and C.~Hervet.
\newblock On the complexity of the data-driven {W}asserstein distributionally
  robust binary problem.
\newblock Technical report, hal-03595342, 2022.

\bibitem{KY97}
P.~Kouvelis and G.~Yu.
\newblock {\em Robust Discrete Optimization and its Applications}.
\newblock Kluwer Academic Publishers, 1997.

\bibitem{PS98}
C.~H. Papadimitriou and K.~Steiglitz.
\newblock {\em Combinatorial optimization: algorithms and complexity}.
\newblock Dover Publications Inc., 1998.

\bibitem{P00}
G.~C. Pflug.
\newblock Some remarks on the {V}alue-at-{R}isk and the {C}onditional
  {V}alue-at-{R}isk.
\newblock In S.~P. Uryasev, editor, {\em Probabilistic {C}onstrained
  {O}ptimization: {M}ethodology and {A}pplications}, pages 272--281. Kluwer
  Academic Publishers, 2000.

\bibitem{RU00}
R.~T. Rockafellar and S.~P. Uryasev.
\newblock Optimization of conditional value-at-risk.
\newblock {\em The Journal of Risk}, 2:21--41, 2000.

\bibitem{WYSZ19}
Z.~Wang, K.~You, S.~Song, and Y.~Zhang.
\newblock {W}asserstein distributionally robust shortest path problem.
\newblock {\em European Journal of Operational Research}, 284:31--43, 2020.

\bibitem{WKS14}
W.~Wiesemann, D.~Kuhn, and M.~Sim.
\newblock Distributionally robust convex optimization.
\newblock {\em Operations Research}, 62:1358--1376, 2014.

\bibitem{YA88}
R.~R. Yager.
\newblock On ordered weighted averaging aggregation operators in multi-criteria
  decision making.
\newblock {\em IEEE Transactions on Systems, Man and Cybernetics}, 18:183--190,
  1988.

\end{thebibliography}

\appendix

\section{Appendix}
\label{dod}  

\begin{proof}[Proof of Lemma~\ref{lemworst}]
The proof is similar in spirit to~\cite[Lemma~2]{WYSZ19}.
Since the function $t + \frac{1}{\alpha} \int_{\Xi}[\pmb{\xi}^T\pmb{x}-t]_{+}\mathbb{P}({\rm d}\pmb{\xi})$ is convex in $t$ and linear in $\mathbb{P}({\rm d}\pmb{\xi})$ we can write:
\begin{align}
  \sup_{\mathbb{P}\in \mathbb{B}_{\epsilon}(\widehat{\mathbb{P}}^N)} {\rm CVaR}^{\alpha}_{\mathbb{P}}[\pmb{\xi}^T\pmb{x}] & =\sup_{\mathbb{P}\in \mathbb{B}_{\epsilon}(\widehat{\mathbb{P}}^N)}
   \inf\left\{t \in \Rset : t + \frac{1}{\alpha} \int_{\Xi}[\pmb{\xi}^T\pmb{x}-t]_{+}\mathbb{P}({\rm d}\pmb{\xi})\right\} \nonumber\\
&  =\inf \left\{t  \in \Rset : t + \frac{1}{\alpha} \sup_{\mathbb{P}\in \mathbb{B}_{\epsilon}(\widehat{\mathbb{P}}^N)}\int_{\Xi}[\pmb{\xi}^T\pmb{x}-t]_{+}\mathbb{P}({\rm d}\pmb{\xi})\right\}. \label{lb2}
\end{align}
If $\mathbb{P}\in \mathbb{B}_{\epsilon}(\widehat{\mathbb{P}}^N)$, then $\mathbb{P}({\rm d}\pmb{\xi})=\sum_{i=1}^N  \Pi({\rm d}\pmb{\xi},\hat{\pmb{\xi}}_i)$ and we can express the inner optimization problem in~(\ref{lb2}) as 

\begin{equation}
\label{lb3}
	\begin{array}{lll}
		 v^*=\sup & \displaystyle \int_{\Xi} \sum_{i=1}^N [\pmb{\xi}^T\pmb{x}-t]_{+}\Pi({\rm d}\pmb{\xi},\hat{\pmb{\xi}}_i)\\
				 & \displaystyle \int_{\Xi} \Pi({\rm d}\pmb{\xi},\hat{\pmb{\xi}}_i)=\frac{1}{N}, & i\in [N],\\
				 &\displaystyle  \int_{\Xi} \sum_{i=1}^N ||\pmb{\xi}-\hat{\pmb{\xi}}_i||_{q} \Pi({\rm d}\pmb{\xi},\hat{\pmb{\xi}}_i)\leq \epsilon,\\
	& \Pi({\rm d}\pmb{\xi},\hat{\pmb{\xi}}_i)\geq 0, & i\in [N].
	\end{array}
\end{equation}
%
The dual to~(\ref{lb3}) is
$$
	\begin{array}{lll}
			u^*=\inf & \displaystyle \frac{1}{N}\sum_{i=1}^N s_i+\lambda\epsilon \\
				& [\pmb{\xi}^T\pmb{x}-t]_{+}-\lambda||\pmb{\xi}-\hat{\pmb{\xi}}_i||_{q} \leq s_i, & \forall\pmb{\xi}\in \Xi, i\in [N],\\
				&\lambda\geq 0,
	\end{array}
$$
where $\lambda\geq 0$ and $s_1,\dots,s_N\in \mathbb{R}$ are Lagrangian multipliers,
which can be equivalently rewritten as follows
$$u^*=\inf_{\lambda \geq 0} \sup_{\pmb{\xi}_1,\dots,\pmb{\xi}_N\in \Xi} \frac{1}{N}\sum_{i=1}^N ([\pmb{\xi}_i^T\pmb{x}-t]_{+}-\lambda||\pmb{\xi}_i-\hat{\pmb{\xi}}_i||_{q})+\lambda \epsilon=$$

\begin{equation}
\label{lb4}
\inf_{\lambda \geq 0} \sup_{\pmb{\xi}_1,\dots,\pmb{\xi}_N\in \Xi} \frac{1}{N}\sum_{i=1}^N [\pmb{\xi}_i^T\pmb{x}-t]_{+}-\lambda \left(\frac{1}{N}\sum_{i=1}^N ||\pmb{\xi}_i-\hat{\pmb{\xi}}_i||_{q}- \epsilon\right).
\end{equation}
Problem~(\ref{lb4}) is equivalent to

$$
	\begin{array}{lll}
			 \displaystyle u^*=\sup & \displaystyle  \frac{1}{N}\sum_{i=1}^N [\pmb{\xi}_i^T\pmb{x}-t]_{+} \\
				 & \displaystyle \frac{1}{N}\sum_{i=1}^N ||\pmb{\xi}_i-\hat{\pmb{\xi}}_i||_{q}\leq \epsilon \\
				 & \pmb{\xi}_1,\dots,\pmb{\xi}_N\in \Xi
	\end{array}
$$
The set of feasible solutions to this problem is closed and bounded, and the objective function is continuous. Therefore, the supremum is attained. Furthermore, the set of feasible solutions is exactly $\mathcal{B}_{\epsilon}(\widehat{\mathbb{P}}^N)$. Hence

\begin{equation}
\label{lb5}
u^*=\max_{(\pmb{\xi}_1,\dots,\pmb{\xi}_N)\in \mathcal{B}_{\epsilon}(\widehat{\mathbb{P}}^N)}\frac{1}{N}\sum_{i=1}^N [\pmb{\xi}_i^T\pmb{x}-t]_{+}.
\end{equation}
By weak duality, we get $u^*\geq v^*$ and by~(\ref{lb2}) we have
\begin{equation}
\label{lb6}
\sup_{\mathbb{P}\in \mathbb{B}_{\epsilon}(\widehat{\mathbb{P}}^N)} {\rm CVaR}^{\alpha}_{\mathbb{P}}[\randxiT\pmb{x}]\leq \inf\left\{t \in \Rset : t+\frac{1}{\alpha} \max_{\mathbb{P}\in \mathcal{B}_{\epsilon}(\widehat{\mathbb{P}}^N)}\frac{1}{N}\sum_{i=1}^N [\pmb{\xi}_i^T\pmb{x}-t]_{+} \right\}=\max_{\mathbb{P}\in \mathcal{B}_{\epsilon}(\widehat{\mathbb{P}}^N)} {\rm CVaR}^{\alpha}_{\mathbb{P}}[\randxiT\pmb{x}].
\end{equation}
It is easy to see that $\mathcal{B}_{\epsilon}(\widehat{\mathbb{P}}^N)\subseteq  \mathbb{B}_{\epsilon}(\widehat{\mathbb{P}}^N)$. Indeed, for $\mathbb{P}^N=(\pmb{\xi}_1,\dots,\pmb{\xi}_N)\in  \mathcal{B}_{\epsilon}(\widehat{\mathbb{P}}^N)$, consider the joint probability distribution $\Pi=((\pmb{\xi}_1,\hat{\pmb{\xi}}_1),\dots,(\pmb{\xi}_N,\hat{\pmb{\xi}}_N))$ in $\Xi \times \Xi$. Then 
$$\int_{\Xi \times \Xi} ||\pmb{\xi}_1-\pmb{\xi}_2||_q \Pi({\rm d} \pmb{\xi}_1, {\rm d}\pmb{\xi}_2)=\frac{1}{N}\sum_{i=1}^N ||\pmb{\xi}_i-\hat{\pmb{\xi}}_i||_{q}\leq \epsilon$$
and $d_W(\widehat{\mathbb{P}}^N, \mathbb{P}^N)\leq \epsilon$. Hence, and by~(\ref{lb6}), we obtain
$$
\sup_{\mathbb{P}\in \mathbb{B}_{\epsilon}(\widehat{\mathbb{P}}^N)} {\rm CVaR}^{\alpha}_{\mathbb{P}}[\randxiT\pmb{x}]=\max_{\mathbb{P}\in \mathcal{B}_{\epsilon}(\widehat{\mathbb{P}}^N)} {\rm CVaR}^{\alpha}_{\mathbb{P}}[\randxiT\pmb{x}].
$$
\end{proof}

\begin{proof}[Proof of Theorem~\ref{propunr}]
Assume that $\alpha\in (\frac{l-1}{N},\frac{l}{N}]$ for some $l\in [N]$.
For the support set $\Xi=\Rset^n_{+}$ we get 
\begin{equation}
\label{defbu}
\mathcal{B}_{\epsilon}(\widehat{\mathbb{P}}^N)=\{(\hat{\pmb{\xi}}_1+\pmb{\Delta}_1,\dots, \hat{\pmb{\xi}}_N+\pmb{\Delta}_N): \sum_{i=1}^N||\pmb{\Delta}_i||_{q}\leq N\epsilon,\; \pmb{\Delta}_i\in \mathbb{R}^n, i\in [N]\}.
\end{equation}
Let us fix $\pmb{x}\in \mathcal{X}$ and assume that
$$\hat{\pmb{\xi}}^T_{\sigma(1)}\pmb{x}\geq \hat{\pmb{\xi}}^T_{\sigma(2)}\pmb{x}\geq \dots\geq \hat{\pmb{\xi}}^T_{\sigma(N)}\pmb{x}.$$
Let $\mathbb{P}^N=(\hat{\pmb{\xi}}_1+\pmb{\Delta}_1,\dots,\hat{\pmb{\xi}}_N+\pmb{\Delta}_N)\in \mathcal{B}_{\epsilon}(\widehat{\mathbb{P}}^N)$ be a worst probability distribution for $\pmb{x}$ such that 
$(\hat{\pmb{\xi}}_{\sigma'(1)}+\pmb{\Delta}_{\sigma'(1)})^T\pmb{x}\geq\dots\geq (\hat{\pmb{\xi}}_{\sigma'(N)}+\pmb{\Delta}_{\sigma'(N)})^T\pmb{x}$.
Then, using Proposition~\ref{propowa}
\begin{equation}
\label{h0}
{\rm CVaR}^{\alpha}_{\mathbb{P}^N}[\randxiT\pmb{x}]=\sum_{i=1}^N w_i (\hat{\pmb{\xi}}_{\sigma'(i)}+\pmb{\Delta}_{\sigma'(i)})^T\pmb{x}\leq \sum_{i=1}^N w_i (\hat{\pmb{\xi}}_{\sigma(i)}+\pmb{\Delta}_{\sigma'(i)})^T\pmb{x}\leq {\rm CVaR}^{\alpha}_{\mathbb{P}^{*N}}[\randxiT\pmb{x}],
\end{equation}
where $\mathbb{P}^{*N}$ is the probability distribution,  in which
 $\pmb{\xi}^*_{\sigma(i)}=\hat{\pmb{\xi}}_{\sigma(i)}+\pmb{\Delta}_{\sigma'(i)}$, $i\in [N]$. The last inequality in~(\ref{h0}) follows from the fact that $w_1\geq w_2\geq\dots\geq w_N$. It is easy to see that  $\mathbb{P}^{*N}\in\mathcal{B}_{\epsilon}(\widehat{\mathbb{P}}^N)$ and, according to~(\ref{h0}), $\mathbb{P}^{*N}$ is also a worst probability distribution for $\pmb{x}$.
 Therefore, we get
\begin{equation}
\label{e0002}
\max_{\mathbb{P}^N\in\mathcal{B}_{\epsilon}(\widehat{\mathbb{P}}^N)}{\rm CVaR}^{\alpha}_{\mathbb{P}^N}[\randxiT\pmb{x}]= \left\{
\begin{array}{llll}
		\max & \displaystyle \sum_{i=1}^N w_i(\hat{\pmb{\xi}}_{\sigma(i)}+\pmb{\Delta}_{i})^T\pmb{x} \\
		&\displaystyle \sum_{i=1}^N ||\pmb{\Delta}_i||_{q} \leq \epsilon N,\\
		& \pmb{\Delta}_i\in \Rset^n, & i\in [N].
\end{array} \right.
\end{equation}
 Since, by Proposition~\ref{propowa}, $\sum_{i=1}^N w_i \hat{\pmb{\xi}}_{\sigma(i)}^T\pmb{x}={\rm CVaR}^{\alpha}_{\widehat{\mathbb{P}}^N}[\randxiT\pmb{x}]$,
\begin{equation}
\label{g0}
\max_{\mathbb{P}^N\in\mathcal{B}_{\epsilon}(\widehat{\mathbb{P}}^N)}{\rm CVaR}^{\alpha}_{\mathbb{P}^N}[\randxiT\pmb{x}]={\rm CVaR}^{\alpha}_{\widehat{\mathbb{P}}^N}[\randxiT\pmb{x}]+\left\{ \begin{array}{llll}
				\max  & \displaystyle \sum_{i=1}^{N} w_i\pmb{\Delta}_{i}^T\pmb{x}\\
				& \displaystyle \frac{1}{N}\sum_{i=1}^N ||\pmb{\Delta}_i||_{q} \leq \epsilon,\\
				&\pmb{\Delta}_i\in \Rset^n, & i\in [N].
				\end{array}\right.			
\end{equation}
Let us focus on the inner maximization problem in~(\ref{g0}). The Lagrangian function with Lagrangian multiplier $v\geq 0$ for this problem takes the following form:
$$
L(\pmb{\Delta}_1,\dots,\pmb{\Delta}_N,v)= \sum_{i=1}^{N} w_i \pmb{\Delta}_{i}^T\pmb{x}-v(\frac{1}{N}\sum_{i=1}^N ||\pmb{\Delta}_i||_{q} - \epsilon)=
$$

$$
\sum_{i=1}^{l} (w_i \pmb{\Delta}_{i}^T\pmb{x}-v\frac{1}{N}||\pmb{\Delta}_i||_{q})-v\frac{1}{N}\sum_{i=l+1}^N ||\pmb{\Delta}_i||_{q}+v\epsilon,
$$
which follows from the fact that $w_{l+1}=\dots=w_N=0$.
We now show that
\begin{equation}
\label{e0003}
\sup_{\pmb{\Delta}_i} (w_i\pmb{\Delta}_i^T\pmb{x}-v\frac{1}{N}||\pmb{\Delta}_i||_{q})=\left\{\begin{array}{lll} 
					0 & \text{if} & Nw_i||\pmb{x}||_{q'}\leq v,\\
					\infty & \text{if} & Nw_i||\pmb{x}||_{q'}> v.
			\end{array}\right.
\end{equation}
\noindent Indeed, from the definition of the dual norm (see, e.g.,~\cite{BV04}), we get for each $t\geq 0$
$$t||\pmb{x}||_{q'}=\sup\{\pmb{\Delta}^T_i\pmb{x}: ||\pmb{\Delta}_i||_q\leq t\}$$
and, because $v\geq 0$
$$\sup_{\pmb{\Delta}_i} (w_i\pmb{\Delta}_i^T\pmb{x}-v\frac{1}{N}||\pmb{\Delta}_i||_{q})=\sup_{t\geq 0}\sup_{||\pmb{\Delta}_i||_{q}\leq t} (w_i\pmb{\Delta}_i^T\pmb{x}-v\frac{1}{N}t)=$$
$$\sup_{t\geq 0} (w_i t ||\pmb{x}||_{q'}-v\frac{1}{N}t)=\sup_{t\geq 0} t (w_i  ||\pmb{x}||_{q'}-v\frac{1}{N}),$$
which implies~(\ref{e0003}). Hence
$$\inf_v\sup_{\pmb{\Delta}_1,\dots,\pmb{\Delta}_N}L(\pmb{\Delta}_1,\dots,\pmb{\Delta}_N,v)=\left\{
\begin{array}{lll}
	v\epsilon & \text{if} & Nw_i||\pmb{x}||_{q'} \leq v,\; i\in [l],\\
	\infty & \text{otherwise}.
\end{array}\right.
$$
The optimization problem in~(\ref{g0}) is a convex one that satisfies the Slater condition for each $\epsilon>0$.  Thus, by strong duality
$${\rm CVaR}^{\alpha}_{\mathbb{P}^N}[\randxiT\pmb{x}]={\rm CVaR}^{\alpha}_{\widehat{\mathbb{P}}^N}[\randxiT\pmb{x}]+\max_{i\in [l]} Nw_i\epsilon||\pmb{x}||_{q'}.$$
Finally, using the definition of the weights in Proposition~\ref{propowa}, we obtain
${\rm CVaR}^{\alpha}_{\widehat{\mathbb{P}}^N}[\pmb{\xi}^T\pmb{x}]={\rm CVaR}^{\alpha}_{\widehat{\mathbb{P}}^N}[\pmb{\xi}^T\pmb{x}]+
		N\epsilon||\pmb{x}||_{q'}$ if $\alpha\leq \frac{1}{N}$ and ${\rm CVaR}^{\alpha}_{\mathbb{P}^N}[\randxiT\pmb{x}]={\rm CVaR}^{\alpha}_{\widehat{\mathbb{P}}^N}[\randxiT\pmb{x}]+
		\frac{1}{\alpha}\epsilon||\pmb{x}||_{q'}$, otherwise. This yields~(\ref{mincvd2}).

\end{proof}

\end{document}